\documentclass[10pt, notitlepage]{article}
\usepackage{amssymb,amsmath,comment}
\usepackage{amsthm}
\usepackage[pagewise]{lineno}
\catcode`\@=11 \@addtoreset{equation}{section}
\def\thesection{\arabic{section}}

\def\theequation{\thesection.\arabic{equation}}
\catcode`\@=12
\usepackage{colortbl}
\usepackage{hyperref}
\usepackage[mathscr]{eucal}
\usepackage{epsf}
\usepackage{esint}
\usepackage{a4wide}

\newcommand{\e}{\varepsilon}

\newcommand{\Om} {\Omega}

\newcommand{\la} {\lambda}

\newcommand{\na} {\nabla}

\newcommand{\mb} {\mathbb}
\newcommand{\mc} {\mathcal}

\usepackage[all]{xy}
\catcode`\@=11
\def\theequation{\@arabic{\c@section}.\@arabic{\c@equation}}

\newtheorem{Theorem}{Theorem}[section]
\newtheorem{lemma}[Theorem]{Lemma}
\newtheorem{Proposition}[Theorem]{Proposition}

\newtheorem{Remark}[Theorem]{Remark}
\newtheorem{definition}[Theorem]{Definition}

\def\XXint#1#2#3{{\setbox0=\hbox{$#1{#2#3}{\int}$ }
		\vcenter{\hbox{$#2#3$ }}\kern-.6\wd0}}

\usepackage[all]{xy}
\catcode`\@=11
\def\theequation{\@arabic{\c@section}.\@arabic{\c@equation}}
\catcode`\@=12

\begin{document}

{\vspace{0.01in}
\title{Mixed local–nonlocal eigenvalue problems in the Heisenberg group: spectral theory and singular multiplicity}
		\author{Prashanta Garain\footnote{Department of Mathematical Sciences, Indian Institute of Science Education and Research Berhampur, Permanent Campus, At/Po:-Laudigam, Dist.-Ganjam, Odisha, India-760003.
					Email: {\tt pgarain92@gmail.com}\,} \,and Vicen\c{t}iu D. R\u{a}dulescu\footnote{AGH University of Krak\'ow, Faculty of Applied Mathematics,
al. Mickiewicza 30,
30-059 Krak\'ow,
Poland \& Brno University of Technology, Faculty of Electrical Engineering and Communication, Technick\'a 3058/10, Brno
61600, Czech Republic. Email: {\tt radulescu@agh.edu.pl}}}
	\date{}

\maketitle

\begin{abstract}\noindent
This paper investigates nonlinear eigenvalue problems governed by a family of operators that combine two distinct diffusion mechanisms: a classical local diffusion, which accounts for short-range interactions, and a fractional nonlocal diffusion, which captures long-range effects. The analysis is carried out in the Heisenberg group, a non-Euclidean geometric setting in which motion is constrained to a distinguished set of horizontal directions.
As an application of the obtained spectral theory, we establish multiplicity results for perturbed singular problems in both the purely nonlocal and mixed local--nonlocal settings. For the mixed case, a gradient convergence theorem plays a crucial role in deriving these multiplicity results. A key feature of our analysis is the treatment of variable singular exponents, allowing the singularity to vary throughout the domain. 
To the best of our knowledge, this is the first systematic study of nonlinear eigenvalue problems and singular problems with variable singularity exponents for mixed local–nonlocal operators in the Heisenberg group. The results therefore open a new direction in the spectral theory of subelliptic operators and provide a rigorous mathematical foundation for applications ranging from anomalous diffusion and nonlocal phase transitions to image analysis and control theory on nonholonomic systems.


\end{abstract}

\maketitle

\medskip

\tableofcontents

\section{Introduction and main results}
\subsection{Introduction}
Eigenvalue problems for nonlinear differential operators constitute a central topic in the theory of partial differential equations and arise naturally in a wide variety of physical, biological and engineering applications; see, for instance, \cite{DrabekKufnerNicolosi, Lindpams}. They appear in the study of vibration frequencies of elastic bodies, stability of equilibrium states, reaction-diffusion systems, population dynamics, quantum mechanics and image processing; see \cite{DrabekKufnerNicolosi, Henrot}. In particular, the principal eigenvalue often determines critical thresholds, asymptotic behavior of evolution equations, stability properties of solutions and bifurcation phenomena; see \cite{BerestyckiNirenbergVaradhan, Lindpams}. Therefore, the existence and qualitative properties of eigenfunctions play a fundamental role in the analysis of nonlinear partial differential equations and nonlinear boundary value problems.

In recent years, considerable attention has been devoted to nonlinear equations involving nonlocal operators due to their applications in anomalous diffusion, L\'evy processes, phase transitions, finance, image processing and population dynamics; see \cite{CaffarelliSilvestre, DiNezzaPalatucciValdinoci}. Moreover, equations combining local and nonlocal diffusion have attracted increasing interest, as they describe the coexistence of short-range diffusion and long-range interactions in heterogeneous media; see \cite{BiagiDipierroValdinociVecchi2}. These developments naturally lead to the study of spectral properties of both purely nonlocal and mixed local-nonlocal operators.

Compared with the Euclidean framework, the theory of nonlinear eigenvalue problems in sub-Riemannian spaces is still relatively less developed. Among such spaces, the Heisenberg group $\mathbb H^n$, the prototypical example of a Carnot group, plays a distinguished role due to its rich geometric structure and numerous applications; see, for instance, \cite{BonfiglioliLanconelliUguzzoni, Cap}. Besides its fundamental role in harmonic analysis, geometric measure theory and several complex variables, the Heisenberg group provides a natural setting for subelliptic partial differential equations. Its geometry also appears in models involving nonholonomic constraints, such as robotics, motion planning, visual cortex modelling and geometric optics; see \cite{Montgomery}. In this framework, diffusion is constrained to horizontal directions, leading naturally to the horizontal gradient and the nonlinear sub-Laplacian, while their fractional counterparts describe long-range interactions compatible with the intrinsic sub-Riemannian geometry. Hence, eigenvalue problems associated with these operators provide the appropriate spectral framework for nonlinear diffusion processes in the Heisenberg group.

Motivated by these considerations, this paper is devoted to the study of nonlinear $(p,q)$-eigenvalue problems associated with purely nonlocal and mixed local-nonlocal operators in the Heisenberg group, together with their application to singular problems involving variable nonlinearities. More precisely, we first consider the following eigenvalue problem:
\begin{equation}\label{evp}
\mathcal{M}_\alpha\,u:=-\alpha\Delta_{H,p}u+(-\Delta_{H,p})^s u
=\lambda \|u\|_{L^q(\Omega)}^{p-q}|u|^{q-2}u
\quad\text{in }\Omega,\qquad
u=0\quad\text{in }\mathbb H^n\setminus\Omega .
\end{equation}
Here, $\Omega\subset\mathbb H^n$, $n\geq 2$ is a bounded smooth domain, $Q=2n+2$ denotes the homogeneous dimension of $\mathbb H^n$, $\alpha\in\{0,1\}$ and $s\in(0,1)$. The case $\alpha=0$ corresponds to the purely nonlocal fractional $p$-sub-Laplacian, whereas $\alpha=1$ corresponds to the mixed local-nonlocal operator. We assume $(p,q)\in I_\alpha$, where
\begin{equation}\label{Ialpha}
I_\alpha:=
\begin{cases}
\Big(1,\dfrac{Q}{s}\Big)\times(1,p_s^*),
& p_s^*=\dfrac{Qp}{Q-sp}, \quad \text{if }\alpha=0\\[2mm]
(1,Q)\times(1,p^*), 
& p^*=\dfrac{Qp}{Q-p}, \quad \text{if }\alpha=1.
\end{cases}
\end{equation}
The operator $\Delta_{H,p}$ is the $p$-sub-Laplacian defined by
\[
\Delta_{H,p}u
=\operatorname{div}\left(|\nabla_H\,u|^{p-2}\nabla_H\,u\right),
\]
and $(-\Delta_{H,p})^s$ denotes the fractional $p$-sub-Laplacian given by
\[
(-\Delta_{H,p})^su(x)
=\mathrm{P.V.}\int_{\mathbb H^n}
\frac{|u(x)-u(y)|^{p-2}(u(x)-u(y))}
{|y^{-1}\circ x|_{K}^{Q+ps}}\,dy,
\]
where P.V. denotes the principal value and $|\cdot|_K$ denotes the Kor\'anyi norm defined in \eqref{kn} below.

The spectral information obtained from \eqref{evp} is then used to study the parameter dependent perturbed singular problem
\begin{equation}\label{meqn}
\mathcal M_\alpha\,u
=\mu\,u^{-\gamma(x)}+u^\beta
\quad\text{in }\Omega,\qquad
u>0\quad\text{in }\Omega,\qquad
u=0\quad\text{on }\mathbb H^n\setminus\Omega ,
\end{equation}
where $\alpha\in\{0,1\}$, $\mu>0$ and $s\in(0,1)$. We assume that $(p,\beta)\in J_\alpha$, where
\begin{equation}\label{Jalpha}
J_\alpha:=
\begin{cases}
\Big(1,\dfrac{Q}{s}\Big)\times(p-1,p_{s}^*-1),
& p_{s}^*=\dfrac{Qp}{Q-sp}, \quad \text{if }\alpha=0,\\[2mm]
(1,Q)\times(p-1,p^{*}-1),
& p^{*}=\dfrac{Qp}{Q-p}, \quad \text{if }\alpha=1.
\end{cases}    
\end{equation}
Here the exponent $\gamma\in C(\overline{\Omega})$ is allowed to vary inside the domain $\Om$, which satisfies $0<\gamma(x)<1$ for every $x\in\overline{\Om}$. For $\alpha=0$ and $\alpha=1$, this respectively corresponds to the purely nonlocal and mixed local-nonlocal singular problems. We prove multiplicity results for weak solutions of \eqref{meqn}.

Before discussing our results, we briefly review the related literature. In the Euclidean setting, eigenvalue problems for nonlinear local operators have been extensively studied; see \cite{Garica, Lindpams} for the case $p=q$ and \cite{DrabekKufnerNicolosi, Ercole, Kawohl, Otani} for the more general $(p,q)$-setting. Eigenvalue problems for purely nonlocal fractional $p$-Laplacian type operators have also been widely investigated; see \cite{BRRbook, Brasco, FelmerQuaasTan, LindgrenLindqvist} for $p=q$ and \cite{Ercole} for the fractional $(p,q)$-case. More recently, mixed local-nonlocal $(p,q)$-eigenvalue problems have been studied in the Euclidean framework; see \cite{PPnon} for $p=q$ and $\alpha=1$ and \cite{Craig, Garainriv, Ahmad} for $p\neq q$ respectively.

In the Heisenberg group, the corresponding theory is still rather limited. Local $(p,q)$-eigenvalue problems have been studied in \cite{GUtraz}, while the purely nonlocal case for $p=q$ was considered in \cite{GKRmathann}. Mixed local-nonlocal eigenvalue problems have only been investigated in the linear case $p=2$; see \cite{Deb}. Thus, nonlinear purely nonlocal and mixed local-nonlocal $(p,q)$-eigenvalue problems in the Heisenberg group remain largely unexplored.

The main results of this paper are twofold. First, we establish the existence of nontrivial eigenfunctions and corresponding eigenvalues for problem \eqref{evp}. Our approach is based on an inverse iteration scheme inspired by the nonlinear inverse power method \cite{Ercole}. Furthermore, by combining recent regularity results for mixed local-nonlocal equations with Stampacchia's truncation method, we obtain boundedness and further regularity properties of the eigenfunctions.

Second, we apply this spectral framework to the singular problem \eqref{meqn}. Here, by singularity, we mean that the nonlinearity on the right-hand side of \eqref{meqn} becomes unbounded near the origin due to the presence of the singular exponent $\gamma$. A further feature of $\gamma$ is its variable nature within the domain. Singular problems have been extensively studied, and a substantial body of literature is available in the Euclidean setting for purely local, purely nonlocal, and mixed local--nonlocal problems, both for constant and variable singular exponents. More precisely, for the constant singular exponent case, we refer to \cite{Arcoya, Bocvpde, CRT, Radu, GST, Lz} for the purely local setting, \cite{CMS, Mukanona1} for the purely nonlocal setting, and \cite{AR, Gjga, GUnon, HH} for the mixed local--nonlocal setting. For variable singular exponents, we refer to \cite{Ajde, BGMade, CMP} for the purely local case, \cite{GMcpaa} for the purely nonlocal case, and \cite{Biroud, GKK} for the mixed local--nonlocal case.

In contrast, relatively little is known about singular problems in non-Euclidean settings, even in the Heisenberg group. In the purely local setting, we refer to \cite{GUaamp} and \cite{GK, Pucciacv, Sahu} for related results. For the purely nonlocal case, we refer to \cite{G1, GKRmathann} and in the mixed local--nonlocal setting, see \cite{G2, Gopus}.

Nevertheless, to the best of our knowledge, perturbed singular problems with variable exponents in non-Euclidean settings remain unexplored, even in the Heisenberg group, for both purely nonlocal and mixed local--nonlocal operators. We note that perturbed singular nonlinearities were considered in \cite{GKRmathann} in the purely nonlocal setting under the assumption that the singular exponent $\gamma\in(0,1)$ is constant. Their analysis relies on the Nehari manifold method. In contrast, in the present article, we allow the singular exponent $\gamma\in(0,1)$ to vary and establish multiplicity results for problem \eqref{meqn}. Our approach combines variational and approximation arguments, inspired by \cite{Arcoya}, and provides a unified framework for treating both the purely nonlocal and mixed local--nonlocal cases. In the mixed local--nonlocal setting, an additional gradient convergence result (see Theorem~\ref{gradcgt}) is required, which plays a crucial role in passing to the limit in the nonlinear local term.

\subsection*{Features of the paper}

This paper initiates the systematic study of nonlinear eigenvalue problems governed by mixed local-nonlocal operators in the Heisenberg group, a setting in which motion is constrained to horizontal directions and the geometry is genuinely non-Euclidean.

The analysis treats, within a single framework, both purely nonlocal fractional sub-Laplacians and operators that intertwine classical local diffusion with long-range nonlocal effects, thereby capturing the coexistence of short-range and long-range interactions in heterogeneous media.

In this paper, we develop a spectral theory for nonlinear eigenvalue problems in the Heisenberg group, establishing the existence of eigenfunctions, their boundedness, and their strict positivity. These results extend classical spectral theory beyond the Euclidean and Riemannian frameworks.

A distinctive feature of this paper is the treatment of singular nonlinearities whose exponent is allowed to vary throughout the domain, a far more delicate scenario than the constant-exponent case studied previously.
The spectral information obtained is then deployed to prove multiplicity results for perturbed singular problems, yielding at least two distinct positive weak solutions for small values of the perturbation parameter.
The proofs combine variational methods with approximation arguments, providing a unified treatment of both the purely nonlocal and the mixed local-nonlocal cases.

In the mixed setting, a gradient convergence theorem plays an essential role, enabling the passage to the limit in the nonlinear local term, a technical ingredient absent from the purely nonlocal analysis.

To the best of the authors' knowledge, this is the first systematic study of nonlinear eigenvalue problems and singular problems with variable singularity exponents for mixed local--nonlocal operators in the Heisenberg group, thereby opening a new direction in the spectral theory of subelliptic operators.

\subsection*{Notation and assumptions}

Throughout the paper, we use the following notation and standing assumptions, unless otherwise specified. We assume that
\[
\alpha \in \{0,1\}, \qquad 0<s<1<p<Q,
\]
where $Q=2n+2$ is the homogeneous dimension of the Heisenberg group $\mb{H}^n$, $n\geq 2$.

For the $(p,q)$-eigenvalue problem \eqref{evp}, we assume that
\[
(p,q)\in I_\alpha,
\]
where $I_\alpha$ is defined in \eqref{Ialpha}. Similarly, for the singular problem \eqref{meqn}, we assume that
\[
(p,\beta)\in J_\alpha,
\]
where $J_\alpha$ is defined in \eqref{Jalpha}.

Let $\Omega$ be a bounded smooth domain in $\mathbb{H}^n$, with $n\geq 2$. The group operation $\circ$ and the Kor\'anyi-type norm $|\cdot|_K$ are defined in \eqref{gm} and \eqref{kn}, respectively. For $a,b\in\mathbb{R}^k$, we denote by $\langle a,b\rangle$ the standard inner product and by $|a|$ the Euclidean norm in $\mathbb{R}^k$, where $k\geq 1$.

The notation $\omega\Subset\Omega$ means that $\omega$ is compactly contained in $\Omega$, that is,
\[
\omega\subset\overline{\omega}\subset\Omega.
\]
Moreover, by $\|\cdot\|$, we mean the norm $\|\cdot\|_{X_\alpha}$ defined in \eqref{n1} for $\alpha\in\{0,1\}$. We denote by $\|\cdot\|_{\infty}$ the norm $\|\cdot\|_{L^\infty(\Omega)}$.

For $l>1$, its conjugate exponent is denoted by
\[
l'=\frac{l}{l-1}.
\]
For $m>1$, we define the function
\[
J_m(t)=|t|^{m-2}t,\qquad t\in\mathbb{R}.
\]

We also use the notation
\[
d\nu
=
\frac{dx\,dy}{|y^{-1}\circ x|_{K}^{Q+sp}}.
\]

For $k\in\mathbb{R}$, we adopt the standard notation
\[
k^+=\max\{k,0\},\qquad
k^-=\max\{-k,0\},\qquad
k_-=\min\{k,0\}.
\]

If $F$ is a measurable function on a set $S$ and $c,d$ are constants, then the notation
\[
c\leq F\leq d\quad\text{in }S
\]
means that
\[
c\leq F(x)\leq d
\qquad\text{for almost every }x\in S.
\]

Finally, the symbol $C$ denotes a generic positive constant, which may change from line to line or even on the same line. If $C$ depends on the parameters $r_1,r_2,\ldots,r_k$, we explicitly indicate this dependence by writing
\[
C=C(r_1,r_2,\ldots,r_k).
\]

\subsection{Main results}
Our main results related to the eigenvalue problem \eqref{evp} reads as follows:

\begin{Theorem}\label{newthm}
The following properties hold:
\begin{enumerate}
\item[$(a)$] There exists a sequence $\{w_k\}_{k\in\mathbb{N}}\subset X_\alpha\cap L^q(\Omega)$ such that $\|w_k\|_{L^q(\Omega)}=1$ and for every $v\in X_\alpha$, we have 
\begin{align}\label{its}
&\alpha\int_{\Om}|\nabla_H\,w_{k+1}|^{p-2}\nabla_H\,w_{k+1}\nabla_H\,v\,dx+\int_{\mathbb{H}^n}\int_{\mathbb{H}^n}{J_p(w_{k+1}(x)-w_{k+1}(y))(v(x)-v(y))}\,d\nu\\
&\qquad=\mu_k\int_{\Omega}|w_{k}|^{q-2}w_{k}v\,dx,
\end{align}
where for every $n\in\mathbb{N}$, 
\begin{align}\label{subopmin}
\mu_k\geq\lambda^*:=\inf\left\{\|u\|^p:u\in X_\alpha\cap {L^q(\Omega)},\,\|u\|_{L^q(\Omega)}=1\right\}.
\end{align}

\item[$(b)$] Moreover, the sequences $\{\mu_k\}_{k\in\mathbb{N}}$ and $\{\|w_{k+1}\|_{X_\alpha}^{p}\}_{k\in\mathbb{N}}$ given by \eqref{its} are nonincreasing and converge to the same limit $\eta$, which is bounded below by $\lambda^*$. Further, there exists a subsequence $\{k_j\}_{j\in\mathbb{N}}$ such that both $\{w_{k_j}\}_{j\in\mathbb{N}}$ and $\{w_{k_{j+1}}\}_{j\in\mathbb{N}}$ converges in $X_\alpha$ to the same limit $w\in X_\alpha\cap L^q(\Omega)$ with $\|w\|_{L^q(\Omega)}=1$ and $(\mu,w)$ is an eigenpair of \eqref{meqn}.
\end{enumerate}
\end{Theorem}

\begin{Theorem}\label{subopthm1}
Suppose $\{u_k\}_{k\in\mathbb{N}}\subset X_\alpha\cap L^q(\Omega)$ such that $\|u_k\|_{L^q(\Omega)}=1$ and $\lim_{k\to\infty}\|u_k\|_{X_\alpha}^{p}=\lambda^*$, where $\lambda^*$ is defined in \eqref{subopmin} above.
Then there exists a subsequence $\{u_{n_j}\}_{j\in\mathbb{N}}$ which converges weakly in $X_\alpha$ to $u\in X_\alpha\cap L^q(\Omega)$ with $\|u\|_{L^q(\Omega)}=1$ such that  
$
\lambda^*=\|u\|^p.
$
Moreover, $(\lambda^*,u)$ is an eigenpair of \eqref{meqn} and any associated eigenfunction of $\lambda^*$ are precisely the scalar multiple of those vectors at which $\lambda^*$ is reached.
\end{Theorem}

Our final main result concerns the following qualitative properties of the eigenfunctions of \eqref{meqn}.
\begin{Theorem}\label{regthm}
Assume that $\lambda>0$ is an eigenvalue of the problem \eqref{meqn} and $u\in X_\alpha\setminus\{0\}$ is a corresponding eigenfunction. Then $(a)$ $u\in L^\infty(\Omega)$. $(b)$ Moreover, if $u$ is nonnegative in $\Omega$, then $u>0$ in $\Omega$. Further, for every $\omega\Subset\Omega$ there exists a positive constant $c=c(\omega)$ such that $u\geq c>0$ in $\omega$.
\end{Theorem}
Our final main result is the following multiplicity result for the problem \eqref{meqn}.
\begin{Theorem}\label{thm2}
Assume that $\gamma:\overline{\Om}\to(0,1)$ is a continuous function. Then there exists a constant 
$\mu^{*}>0$ such that, for every $\mu \in (0,\mu^{*})$, equation \eqref{meqn} 
admits at least two distinct weak solutions in $X_\alpha$. Moreover, every weak solution of \eqref{meqn} belongs to $L^\infty(\Om)$.
\end{Theorem}

\subsection*{Organization of the paper}

The paper is organized as follows. In Section 2, we introduce the functional framework, notation and assumptions, and recall some preliminary results that will be used throughout the paper. Section 3 is devoted to the nonlinear $(p,q)$-eigenvalue problem, where we establish the existence of eigenfunctions and corresponding eigenvalues and investigate their regularity properties. In Section 4, we apply the developed spectral theory to a perturbed singular problem involving variable singular nonlinearities and prove multiplicity results for weak solutions.

\section{Functional framework and preliminaries}

\subsection{Heisenberg group and functional spaces}

In this subsection, we introduce the basic notation related to the Heisenberg group and define the local and fractional Sobolev spaces associated with the horizontal structure. We also recall the relevant embedding results and properties of these spaces.

The Euclidean space $\mathbb{R}^{2n+1}$, endowed with the group multiplication
\begin{equation}\label{gm}
x \circ y :=
\left(
x_1+y_1,\,
x_2+y_2,\,
\ldots,\,
x_{2n}+y_{2n},\,
\tau + \tau' + \frac{1}{2}\sum_{i=1}^{n}\big(x_i y_{n+i} - x_{n+i} y_i\big)
\right),
\end{equation}
where $x=(x_1,\ldots,x_{2n},\tau)$ and $y=(y_1,\ldots,y_{2n},\tau')\in\mathbb{R}^{2n+1}$, defines the Heisenberg group $\mathbb{H}^n$.

The left-invariant vector fields on $\mathbb{H}^n$ are given by
\[
X_i = \partial_{x_i} - \frac{x_{n+i}}{2}\,\partial_\tau,
\qquad
X_{n+i} = \partial_{x_{n+i}} + \frac{x_i}{2}\,\partial_\tau,
\qquad 1\le i\le n,
\]
and the nontrivial commutator is
\[
T = \partial_\tau = [X_i,\, X_{n+i}]
= X_i X_{n+i} - X_{n+i} X_i,
\qquad 1\le i\le n.
\]
We refer to $X_1, X_2, \ldots, X_{2n}$ as the horizontal vector fields on $\mathbb{H}^n$, and $T$ as the vertical vector field.

The Haar measure on $\mathbb{H}^n$ is equivalent to the Lebesgue measure on $\mathbb{R}^{2n+1}$.  
For a measurable set $E\subset \mathbb{H}^n$, we denote its Lebesgue measure by $|E|$.

For $x=(x_1,\ldots,x_{2n},\tau)$, we define its Kor\'anyi-type norm by  
\begin{equation}\label{kn}
|x|_K
= \left( \left(\sum_{i=1}^{2n} x_i^2 \right)^{2} + \tau^{2} \right)^{1/4}.
\end{equation}
The Carnot-Carath\'eodory distance between two points $x,y\in\mathbb{H}^n$ is defined as the infimum of the lengths of horizontal curves joining them. We define
\[
d(x,y)=|y^{-1}\circ x|_K,
\]
which is equivalent to the Carnot-Carath\'eodory distance. 

The ball of radius $r>0$ centered at $\xi_0$ with respect to $d$ is given by  
\[
B_r(x_0)=\{\xi\in\mathbb{H}^n: |y^{-1}\circ x_0|_{K}<r\}.
\]
When the center is irrelevant or understood from the context, we write simply $B_r := B_r(x_0)$.

The homogeneous dimension of $\mathbb{H}^n$ is $Q = 2n + 2$. Let $1 < p < \infty$ and $\Om \subset \mathbb{H}^n$ with $n\geq 2$ be a bounded smooth domain.
The Sobolev space $HW^{1,p}(\Omega)$ is defined by
\[
HW^{1,p}(\Om)
=
\left\{
u \in L^{p}(\Om) :
|\nabla_H\,u| \in L^{p}(\Om)
\right\},
\]
where $\na_H\,u$ is the horizontal gradient of $u$ defined by  
$
\na_H\,u = (X_1u, X_2u, \ldots, X_{2n}u).
$
and is endowed with the norm
\[
\|u\|_{HW^{1,p}(\Om)}
=
\|u\|_{L^{p}(\Om)}
+
\left\|\nabla_H\,u\right\|_{L^{p}(\Om)}.
\]

Let $0<s<1 < p < \infty$, and suppose $v:\mathbb{H}^n \to \mathbb{R}$ be a measurable function.  
The Gagliardo seminorm of $v$ is defined by
\[
[v]_{HW^{s,p}(\mathbb{H}^n)}
=
\left(
\int_{\mathbb{H}^n}\int_{\mathbb{H}^n}
{|v(x)-v(y)|^{p}}
\,d\nu
\right)^{1/p},
\]
where 
$$
d\nu:=\frac{dx\,dy}{|y^{-1}\circ x|_{K}^{Q+sp}}.
$$

The fractional Sobolev space on the Heisenberg group is then defined as
\[
HW^{s,p}(\mathbb{H}^n)
=
\left\{
v \in L^{p}(\mathbb{H}^n) : [v]_{HW^{s,p}(\mathbb{H}^n)} < \infty
\right\}.
\]
The space $HW^{s,p}(\mathbb{H}^n)$ is endowed with the natural fractional norm
\[
\|v\|_{HW^{s,p}(\mathbb{H}^n)}
=
\left( 
\|v\|_{L^{p}(\mathbb{H}^n)}^{p}
+
[v]_{HW^{s,p}(\mathbb{H}^n)}^{p}
\right)^{1/p}.
\]

The fractional Sobolev space $HW^{s,p}(\Om)$ and its associated norm  
$\|v\|_{HW^{s,p}(\Om)}$ are defined in an analogous manner.

For $\alpha\in\{0,1\}$, we define the space $X_\alpha$ as the closure of $C_c^{\infty}(\Om)$ with respect to the norm
\begin{equation}\label{n1}
\|u\|_{X_\alpha}:=\Big(\alpha[u]_{1,p}^p+[u]_{s,p}^p\Big)^\frac{1}{p},
\end{equation}
where
$$
[u]_{1,p}:=\|\na_H\,u\|_{L^p(\Om)}\text{ and }
[u]_{s,p}:=[u]_{HW^{s,p}(\mb{H}^n)}.
$$

By \cite[Theorem 2.4]{Deb}, there exists a constant $c=c(Q,p,s,\Omega)>0$ such that
\begin{equation}\label{locnonsem}
[u]_{s,p}\leq c[u]_{1,p}\quad\forall u\in X_1.
\end{equation}
Using \eqref{locnonsem}, it follows that the norm $\|u\|_{X_1}$ defined by \eqref{n1}
is equivalent to $[u]_{1,p}$.

We observe that for any $\alpha\in\{0,1\}$, the spaces $X_\alpha$ are real separable and uniformly convex Banach space.

Next we define the notion of eigenpair of the problem \eqref{meqn}.
\begin{definition}\label{def}
We say that $(\lambda,u)\in \mathbb{R}\times (X_\alpha\setminus\{0\})$ is an eigenpair of \eqref{meqn} if for every $\phi\in X_\alpha$, we have
\begin{equation}\label{mwksol}
\begin{split}
&\alpha\int_{\Omega}|\nabla_H\, u|^{p-2}\nabla_H\, u\nabla_H\,\phi\,dx+\int_{\mathbb{H}^n}\int_{\mathbb{H}^n}{J_p(u(x)-u(y))(\phi(x)-\phi(y))}\,d\nu\\
&=\lambda\|u\|_{L^q(\Om)}^{p-q}\int_{\Omega}|u|^{q-2}u\phi\,dx.
\end{split}
\end{equation}
\end{definition}
We refer to $\lambda$ as an eigenvalue and $u$ as an eigenfunction of \eqref{meqn} corresponding to the eigenvalue $\lambda$.

\begin{definition}\label{def2}
We say that $u\in X_\alpha$ is a weak solution of \eqref{meqn}, if for every $\omega\Subset\Om$, there exists a positive constant $c(\omega)$ such that $u\geq c>0$ in almost everywhere in $\omega$ and $\phi\in C_c^\infty(\Om)$, we have
\begin{equation}\label{wksoleqn}
\begin{split}
&\alpha\int_{\Omega}|\nabla_H\, u|^{p-2}\nabla_H\, u\nabla_H\,\phi\,dx+\int_{\mathbb{H}^n}\int_{\mathbb{H}^n}J_p(u(x)-u(y))(\phi(x)-\phi(y))\,d\nu\\
&=\mu\int_{\Omega}u^{-\gamma(x)}\phi\,dx+\int_{\Om}u^\beta\phi\,dx.
\end{split}
\end{equation}
\end{definition}

\begin{Remark}\label{rmkdef}
We observe that Lemma \ref{lemb} ensures the above Definitions \ref{def} and \ref{def2} are well stated.
\end{Remark}

\subsection{Auxiliary results}
Here we recall several preliminary results that will be used in the subsequent analysis. For the following result, we refer to \cite[Theorem $9.14$]{MB}.
\begin{Theorem}\label{MBthm}
Let $V$ be a real separable reflexive Banach space and $V^*$ be the dual of $V$. Assume that $A:V\to V^{*}$ is a bounded, continuous, coercive and monotone operator. Then $A$ is surjective, i.e., given any $f\in V^{*}$, there exists $u\in V$ such that $A(u)=f$. If $A$ is strictly monotone, then $A$ is also injective. 
\end{Theorem}

The next result follows from from \cite[Theorem 1]{GKRmathann} for $\alpha=0$ and from \cite[Theorem 8.1]{Koskela} (see also \cite[Theorem 5.27]{Cap}) for $\alpha=1$ respectively.
\begin{lemma}\label{lemb}
The space $X_\alpha$ is continuously and compactly embedded in $L^q(\Omega)$.
\end{lemma}

Using Lemma \ref{lemb}, we have the following result.

\begin{lemma}\label{embd}
For every $u\in X_\alpha$, there exists a positive constant $C=C(r,p,Q,s)$ (for $\alpha=0$) and $C=C(r,p,Q)$ (for $\alpha=1$) such that
\begin{equation}\label{embeqn}
\left(\int_{\Omega}|u|^q\,dx\right)^\frac{1}{q}\leq C|\Omega|^{\frac{1}{q}-\frac{1}{l}}[u]_{X_\alpha}^p,
\end{equation}
where $l=p^*$ if $\alpha=1$ and $l=p_s^{*}$ if $\alpha=0$ respectively.
\end{lemma}

For the following result, see \cite[Lemma B.1]{Stam}.
\begin{lemma}
Let $\phi(t)$, $k_0 \le t < \infty$, be a nonnegative and nonincreasing function such that
\[
\phi(h)\le \frac{C}{(h-k)^l}\,\phi(k)^m,
\qquad h>k>k_0,
\]
where $C,l,m$ are positive constants with $m>1$. Then
$
\phi(k_0+d)=0,
$ where $d$ is given by
$
d^l
=
C\,\phi(k_0)^{\,m-1}\,
2^{\frac{lm}{m-1}}.
$
\end{lemma}

For the following algebraic inequality, we refer to \cite[Lemma $2.1$]{Dama}.

\begin{lemma}\label{alg}
Suppose $1<p<\infty$. Let $k\in\mathbb{N}$ and $a,b\in\mathbb{R}^k$. Then there exists a positive constant $C=C(p)$ such that
\begin{equation*}
\langle |a|^{p-2}a-|b|^{p-2}b, a-b \rangle\geq C(|a|+|b|)^{p-2}|a-b|^2.
\end{equation*}
\end{lemma}

Moreover, we have the following inequality from \cite[Lemma 2.22]{AMPjde}.
\begin{lemma}\label{alg1}
Let $\beta>0$. For every $x,y\geq 0$, one has
\[
(x-y)(x^\beta-y^\beta)
\geq
\frac{4\beta}{(\beta+1)^2}
\left(
x^{\frac{\beta+1}{2}}
-
y^{\frac{\beta+1}{2}}
\right)^2.
\]
\end{lemma}

The following result will be crucial in our arguments of the proof of Theorem \ref{thm2}.

\begin{lemma}\label{evpthm}
There exists an eigenfunction 
$e_{1,\alpha}\in X_\alpha\cap L^{\infty}(\Om)$ corresponding to the eigenvalue 
$\lambda_{1,\alpha}$, which satisfies
\begin{equation}\label{eigeneqn}
 \mathcal{M}_\alpha\,e_{1,\alpha}=\lambda_{1,\alpha}|e_{1,\alpha}|^{p-2}e_{1,\alpha}
 \quad \text{in } \Om, \qquad e_{1,\alpha}>0\text{ in }\Om,\quad 
 e_{1,\alpha}=0 \quad \text{in } \mathbb{H}^n\setminus\Om
\end{equation}
such that for every $\omega\Subset\Omega$, there exists a positive constant $c(\omega)$ with the property $e_{1,\alpha}\geq c(\omega)>0$ on $\omega$.
\end{lemma}
\begin{proof}
For $\alpha=1$, see Theorem \ref{subopthm1} and Theorem \ref{regthm}, and for $\alpha=0$, we refer to \cite[Theorem 2 and Theorem 4]{GKRmathann} and \cite[Theorem 3.7]{GMBruno}.
\end{proof}

The following result follows from \cite[Lemma 3.4]{G1} and \cite[Lemma 3.4]{G2}.

\begin{lemma}\label{auxresult}
Let $\alpha\in\{0,1\}$ and $g \in L^{\infty}(\Omega) \setminus \{0\}$ be a nonnegative function in $\Omega$.  
Then there exists a unique solution 
\[
u \in HW_0^{1,p}(\Omega) \cap L^{\infty}(\Omega)
\] 
to the problem
\begin{equation}\label{approxnew}
\mathcal{M}_\alpha u = g \quad \text{in } \Omega, \qquad
u > 0 \ \text{in } \Omega, \qquad
u = 0 \ \text{in } \mathbb{H}^n \setminus \Omega.
\end{equation}
Moreover, for every $\omega \Subset \Omega$, there exists a constant $C(\omega) > 0$ such that 
$
u \ge C(\omega) \quad \text{in } \omega.
$
\end{lemma}

\begin{lemma}\label{tlem}
Let $u\in X_\alpha$ be a weak solution of problem \eqref{meqn}. Then for every $\phi\in X_\alpha$, the equality \eqref{wksoleqn} holds.
\end{lemma}

\begin{proof}
We prove the result only for $\alpha=0$, since the result for $\alpha=1$ is similar. Let $u\in X_0$ be a weak solution of\eqref{meqn}. Suppose $f(x,u)=\mu\,u^{-\gamma(x)}+u^\beta$. Then for every $\phi\in C_c^{\infty}(\Omega)$, we have
\begin{equation}\label{t1}
\int_{\mb{H}^n}\int_{\mb{H}^n}J_p(u(x)-u(y))(\phi(x)-\phi(y))\,d\nu=\int_{\Omega}f(x,u)\phi\,dx.
\end{equation}
Let $\psi\in X_0$, then there exists a sequence
$\{\psi_k\}_{k\in\mb{N}}\subset C_c^{\infty}(\Omega)$ such that
$
0\leq \psi_k\longrightarrow |\psi|
\text{ strongly in }X_0
$
and pointwise almost everywhere in $\Om$. We observe that
\begin{equation*}
\begin{split}
\int_{\Omega}f(x,u)\psi\,dx\leq
\int_{\Omega}f(x,u)|\psi|\,dx\leq
\liminf_{k\to\infty}
\int_{\Omega}f(x,u)\psi_k\,dx&\leq \liminf_{k\to\infty}
\langle-\mathcal{M}_0\,u,\psi_k\rangle\\&\leq C\|u\|_{X_0}^{p-1}
\lim_{k\to\infty}\|\psi_k\|_{X_0}\\
&\leq
C\|u\|_{X_0}^{p-1}\|\,|\psi|\,||_{X_0}\\
&\leq
C\|u\|_{X_0}^{p-1}\|\psi\|_{X_0},
\end{split}
\end{equation*}
for some constant $C>0$. Let $\phi\in X_0$, then there exists a sequence $\{\phi_k\}_{k\in\mb{N}}\subset C_c^{\infty}(\Om)$ such that $\phi_k\to \phi$ strongly in $X_0$. Then by choosing $\psi=\phi_k-\phi$ in the inequality above, we obtain
\begin{equation}\label{t2}
\lim_{k\to\infty}\int_{\Om}f(x,u)\phi_k\,dx=\int_{\Om}f(x,u)\phi\,dx.
\end{equation}
Moreover, since $\phi_k\to\phi$ strongly in $X_0$, we have
\begin{equation}\label{t3}
\begin{split}
\lim_{k\to\infty}
\int_{\mb{H}^n}\int_{\mb{H}^n}
J_p(u(x)-u(y))(\phi_k(x)-\phi_k(y))\,d\nu&=
\int_{\mb{H}^n}\int_{\mb{H}^n}
J_p(u(x)-u(y))(\phi(x)-\phi(y))\,d\nu. 
\end{split}
\end{equation}
Using \eqref{t2} and \eqref{t3} into \eqref{t1}, the result follows.
\end{proof}

\section{Nonlinear $(p,q)$-eigenvalue problem}
\subsection{Preliminaries}
In this subsection, we establish some preliminary results that are crucial to prove our main results related to the nonlinear $(p,q)$-eigenvalue problem \eqref{evp}. To this end, for $\alpha\in\{0,1\}$, we define the operator $A:X_\alpha\to X_\alpha^*$ by
\begin{equation}
\label{a}
\begin{split}
\langle Av,w\rangle:&=\alpha\int_{\Omega}|\nabla_H \,v|^{p-2}\nabla_H\,v\nabla_H \,w\,dx+\int_{\mathbb{H}^n}\int_{\mathbb{H}^n}{J_p(v(x)-v(y))(w(x)-w(y))}\,d\nu,
\end{split}
\end{equation}
and $B:L^q(\Omega)\to (L^q(\Omega))^*$ by
\begin{equation}\label{b}
\begin{split}
\langle B(v),w\rangle:=\int_{\Omega}|v|^{q-2}vw\,dx.
\end{split}
\end{equation}
The symbols $X_\alpha^*$ and $(L^q(\Omega))^*$ denotes the dual of $X_\alpha$ and $L^q(\Omega)$ respectively. First, we have the following result.
\begin{lemma}\label{newlem}
The following results hold:
\begin{enumerate}
\item[$(i)$] The operators $A$ defined by \eqref{a} and $B$ defined by \eqref{b} are continuous.
\item[$(ii)$] Moreover, $A$ is bounded, coercive and monotone.
\end{enumerate}
\end{lemma}
\begin{proof}
\begin{enumerate}
\item[$(i)$] \textbf{Continuity:} Suppose $\{v_k\}_{k\in\mb{N}}\in X_\alpha$ such that $v_k\to v$ in the norm of $X_\alpha$. Thus, up to a subsequence $\nabla_H\,v_{k}\to \nabla_H \,v$ in $\Om$. We observe that 
\begin{equation}\label{mfd}
\||\nabla_H\,v_{k}|^{p-2}\nabla_H\, v_k\|_{L^\frac{p}{p-1}(\Om)}\leq \|\nabla_H\,v_{k}\|^{p-1}_{L^p(\Omega)}\leq c, 
\end{equation}
for some constant $c>0$, which is independent of $n$. Thus, up to a subsequence, we have
\begin{equation}\label{fc}
|\nabla_H\,v_{n}|^{p-2}\nabla_H\,v_{n}\to |\nabla_H \,v|^{p-2}\nabla_H\, v\text{ weakly in }L^{p'}(\Om).
\end{equation}
Moreover, up to a subsequence, we have
\begin{equation}\label{mffc}
\frac{J_p(v_k(x)-v_k(y))}{|y^{-1}\circ x|_{K}^\frac{Q+ps}{p'}}\to\frac{J_p(v(x)-v(y))}{|y^{-1}\circ x|_{K}^\frac{Q+ps}{p'}} 
\end{equation}
weakly in $L^{p'}(\mathbb{H}^{2n})$. Since, the weak limit is independent of the choice of the subsequence, as a consequence of \eqref{fc} and \eqref{mffc}, we have 
$$
\lim_{k\to\infty}\langle Av_k,w\rangle=\langle Av,w\rangle
$$
for every $w\in X_\alpha$. Thus $A$ is continuous. Similarly, we obtain $B$ is continuous.

\item[$(ii)$] \textbf{Boundedness:} Let $\alpha=1$.
Using Cauchy-Schwartz and H\"older's inequality, we observe that
\begin{equation}\label{mest}
\begin{split}
\langle Av,w\rangle&=\int_{\Om}|\nabla_H\,v|^{p-2}\nabla_H\, v\nabla_H\, w\,dx\\
&\quad\quad+\int_{\mathbb{H}^n}\int_{\mathbb{H}^n}{|v(x)-v(y)|^{p-2}(v(x)-v(y))(w(x)-w(y))}\,d\nu\\
&\leq\int_{\Om}|\nabla_H\, v|^{p-1}|\nabla_H\, w|\,dx+\int_{\mathbb{H}^n}\int_{\mathbb{H}^n}{|v(x)-v(y)|^{p-1}|w(x)-w(y)|}\,d\nu\\
&\leq\Big(\int_{\Om}|\nabla_H\, v|^p\,dx\Big)^\frac{p-1}{p}\Big(\int_{\Om}|\nabla_H\, w|^p\,dx\Big)^\frac{1}{p}\\
&\quad\quad+\Big(\int_{\mathbb{H}^n}\int_{\mathbb{H}^n}{|v(x)-v(y)|^p}\,d\nu\Big)^\frac{p-1}{p}\Big(\int_{\mathbb{H}^n}\int_{\mathbb{H}^n}{|w(x)-w(y)|^p}\,d\nu\Big)^\frac{1}{p}\\
&\leq\Bigg[\Big(\int_{\Om}|\nabla_H\, v|^p\,dx\Big)^\frac{p-1}{p}+\Big(\int_{\mathbb{H}^n}\int_{\mathbb{H}^n}{|v(x)-v(y)|^p}\,d\nu\Big)^\frac{p-1}{p}\Bigg]\|w\|_{X_\alpha}\\
&\leq\Big(\int_{\Om}|\nabla_H\, v|^p\,dx+\int_{\mathbb{H}^n}\int_{\mathbb{H}^n}{|v(x)-v(y)|^p}\,d\nu\Big)^\frac{p-1}{p}\|w\|_{X_\alpha}=\|v\|_{X_\alpha}^{p-1}\|w\|_{X_\alpha}.
\end{split}
\end{equation}
Therefore, we have
$$
\|Av\|_{X_\alpha^*}=\sup_{\|w\|_{X_\alpha}\leq 1}|\langle Av,w\rangle|\leq\|v\|_{X_\alpha}^{p-1}\|w\|_{X_\alpha}\leq\|v\|^{p-1}_{X_\alpha}.
$$
The same estimate above can be proved when $\alpha=0$ by exactly in the same way.
Thus, $A$ is bounded. 

\noindent
\textbf{Coercivity:}  We observe that 
$$
\langle Av,v\rangle=\alpha\int_{\Om}|\nabla_H\, v|^p\,dx+\int_{\mathbb{H}^n}\int_{\mathbb{H}^n}{|v(x)-v(y)|^p}\,d\nu=\|v\|_{X_\alpha}^p.
$$
Since $p>1$, we have $A$ is coercive.\\

\noindent
\textbf{Monotonicity:} We recall the notation for $d\nu$ and $J_p$. Thus, for every $v,w\in X_\alpha$, using Lemma \ref{alg}, we have
\begin{equation*}
\begin{split}
\langle Av-Aw,v-w\rangle&
=\alpha\int_{\Om}\langle|\nabla_H\, v|^{p-2}\nabla_H\, v-|\nabla_H\,w|^{p-2}\nabla_H\,w,\nabla_H\,(v-w)\rangle\,dx\\
&\quad+\int_{\mathbb{H}^n}\int_{\mathbb{H}^n}\Big(J_p(v(x)-v(y))-J_p(w(x)-w(y))\Big)((v(x)-v(y))-(w(x)-w(y)))\,d\nu\\
&\qquad\geq 0.
\end{split}
\end{equation*}
Therefore, $A$ is monotone.
\end{enumerate}
\end{proof}

\begin{lemma}\label{auxlmab}
The operators $A$ defined by \eqref{a} and $B$ defined by \eqref{b} satisfy the following properties: 
\begin{enumerate}
\item[$(H_1)$] $A(tv)=|t|^{p-2}tA(v)\quad\forall t\in\mathbb{R}\quad \text{and}\quad\forall v\in X_\alpha$.

\item[$(H_2)$] $B(tv)=|t|^{q-2}tB(v)\quad\forall t\in\mathbb{R}\quad \text{and}\quad\forall v\in L^q(\Omega)$.

\item[$(H_3)$] $\langle A(v),w\rangle\leq\|v\|_{X_\alpha}^{p-1}\|w\|_{X_\alpha}$ for all $v,w\in X_\alpha$, where the equality holds if and only if $v=0$ or $w=0$ or $v=t w$ for some $t>0$.

\item[$(H_4)$] $\langle B(v),w\rangle\leq\|v\|_{L^q(\Omega)}^{q-1}\|w\|_{L^q(\Omega)}$ for all $v,w\in {L^q(\Omega)}$, where the equality holds if and only if $v=0$ or $w=0$ or $v=t w$ for some $t\geq 0$.

\item[$(H_5)$] For every $w\in L^q(\Omega)\setminus\{0\}$ there exists $u\in X_\alpha\setminus\{0\}$ such that
$$
\langle A(u),v\rangle=\langle B(w),v\rangle\quad\forall\quad v\in X_\alpha.
$$
\end{enumerate}
\end{lemma}
\begin{proof}
\begin{enumerate}
\item[$(H_1)$] Follows by the definition of $A$.

\item[$(H_2)$] Follows by the definition of $B$.

\item[$(H_3)$] We prove the result only for $\alpha=1$, since the proof is similar when $\alpha=0$. First, we note that from \eqref{mest} the inequality
$\langle A(v),w\rangle\leq\|v\|_{X_\alpha}^{p-1}\|w\|_{X_\alpha}$ holds for all $v,w\in X_\alpha$. Let the equality
\begin{equation}\label{mequal}
\langle A(v),w\rangle=\|v\|_{X_\alpha}^{p-1}\|w\|_{X_\alpha}
\end{equation}
holds for every $v,w\in X_\alpha$. We claim that either $v=0$ or $w=0$ or $v=tw$ for some constant $t>0$. Indeed, if $v=0$ or $w=0$, then \eqref{mequal} is true. Therefore, we assume that both $v$ and $w$ are not identically zero in $\Om$ and prove that $v=dw$ for some constant $d>0$. By the estimate \eqref{mest} if the equality \eqref{mequal} holds, then we have
\begin{equation}\label{mequal1}
\langle A(v), w\rangle=\int_{\Om}|\nabla_H\,v|^{p-1}|\nabla_H\,w|\,dx+\int_{\mathbb{H}^n}\int_{\mathbb{H}^n}{|v(x)-v(y)|^{p-1}|w(x)-w(y)|}\,d\nu,
\end{equation}
which gives us
\begin{equation}\label{mCS}
\int_{\Om}f(x)\,dx+\int_{\mathbb{H}^n}\int_{\mathbb{H}^n}g(x,y)\,d\nu=0,
\end{equation}
where 
$$
f(x)=|\nabla_H\,v|^{p-1}|\nabla_H\,w|-|\nabla_H\,v|^{p-2}\nabla_H\,v\nabla_H\,w
$$
and
$$
g(x,y)=|v(x)-v(y)|^{p-1}|w(x)-w(y)|-|v(x)-v(y)|^{p-2}(v(x)-v(y))w(x)-w(y).
$$
By Cauchy-Schwartz inequality, we have $f\geq 0$ in $\Om$ and $g\geq 0$ in $\mathbb{H}^n\times\mathbb{H}^n$. Hence using these facts in \eqref{mCS}, we have $f=0$ in $\Om$, which reduces to
\begin{equation}\label{mfCS}
|\nabla_H\,v|^{p-1}|\nabla_H\,w|=|\nabla_H\,v|^{p-2}\nabla_H\,v\nabla_H\,w\text{ in }\Om.
\end{equation}
Hence $\nabla_H\, v(x)=c(x)\nabla_H\, w(x)$ for some $c(x)\geq 0$.
On the other hand, if the equality \eqref{mequal} holds, then by the estimate \eqref{mest} we have
\begin{equation}\label{mequal2}
\begin{split}
f_1-f_2=g_2-g_1,
\end{split}
\end{equation}
where
$$
f_1=\int_{\Om}|\nabla_H\,v|^{p-1}|\nabla_H\,w|\,dx,\quad f_2=\Big(\int_{\Om}|\nabla_H \,v|^p\,dx\Big)^\frac{p-1}{p}\Big(\int_{\Om}|\nabla_H \,w|^p\,dx\Big)^\frac{1}{p},
$$
$$
g_1=\int_{\mathbb{H}^n}\int_{\mathbb{H}^n}{|v(x)-v(y)|^{p-2}(v(x)-v(y))(w(x)-w(y))}\,d\nu
$$
and
$$
g_2=\Big(\int_{\mathbb{H}^n}\int_{\mathbb{H}^n}{|v(x)-v(y)|^p}\,d\nu\Big)^\frac{p-1}{p}\Big(\int_{\mathbb{H}^n}\int_{\mathbb{H}^n}{|w(x)-w(y)|^p}\,d\nu\Big)^\frac{1}{p}.
$$
By the H\"older's inequality, we know that $f_1-f_2\leq 0$ and $g_2-g_1\geq 0$. Therefore, we obtain from \eqref{mequal2} that
$$
f_1=f_2\text{ and }g_1=g_2.
$$
Since $f_1=f_2$, the equality in H\"older's inequality holds, which gives
\begin{equation}\label{mCSeq2}
|\nabla_H\,v|=d\,|\nabla_H\,w|\text{ in }\Om,
\end{equation}
for some constant $d>0$. Therefore, $c(x)=d$. Therefore, using Lemma \ref{lemb} with $q=p$, we obtain $v=d\,w$ a.e. in $\Om$ for some constant $d>0$. Hence, the property $(H3)$ is verified. 

\item[$(H_4)$] This property can be verified similarly as in $(H_3)$.

\item[$(H_5)$] Since $X_\alpha$ is a real separable and reflexive Banach space. By Lemma \ref{newlem}, the operator $A:X_\alpha\to X_\alpha^*$ is bounded, continuous, coercive and monotone.

By the Sobolev embedding theorem, we have $X_\alpha$ is continuously embedded in $L^q(\Omega)$. Therefore, $B(w)\in X_\alpha^*$ for every $w\in L^q(\Omega)\setminus\{0\}$.

Hence, by Theorem \ref{MBthm}, for every $w\in L^q(\Omega)\setminus\{0\}$, there exists $u\in X_\alpha\setminus\{0\}$ such that
$$
\langle A(u),v\rangle=\langle B(w),v\rangle\quad\forall v\in X_\alpha.
$$
Hence the property $(H_5)$ holds. This completes the proof.
\end{enumerate}
\end{proof}

\subsection{Proof of the eigenvalue problem related main results:}
\begin{proof}(Proof of Theorem \ref{newthm})
\begin{enumerate}
\item[$(a)$] First we recall the definition of the operators $A:X_\alpha\to X_\alpha^*$ from \eqref{a} and $B:L^q(\Omega)\to (L^{q}(\Omega))^*$ from \eqref{b} respectively. Then, noting the property $(H_5)$ from Lemma \ref{auxlmab} and proceeding along the lines of the proof in \cite[page $579$ and pages $584-585$]{Ercole}, the result follows.

\item[$(b)$] Taking into account that $X_\alpha$ is uniformly convex Banach space and by Lemma \ref{lemb}, $X_\alpha$ is compactly embedded in $L^q(\Omega)$. Next, using Lemma \ref{newlem}-$(i)$, the operators $A:X_\alpha\to X_\alpha^*$ and $B:L^q(\Omega)\to (L^q(\Omega))^*$ are continuous and by Lemma \ref{auxlmab}, the properties $(H_1)-(H_5)$ holds. Noting these facts, the result follows from \cite[page $579$, Theorem 1]{Ercole}.
\end{enumerate}
\end{proof}

\begin{proof}(Proof of Theorem \ref{subopthm1})
The proof follows due to the same reasoning as in the proof of Theorem \ref{newthm}-$(b)$ except that here we apply \cite[page $583$, Proposition $2$]{Ercole} in place of \cite[page $579$, Theorem 1]{Ercole}.
\end{proof}

\begin{proof}(Proof of Theorem \ref{regthm})
We prove the result only for $\alpha=1$, since the proof for $\alpha=0$ is analogous by taking into account Lemma \ref{embd}.
\begin{enumerate}
\item[$(a)$] Let $\alpha=1$. Due to the homogeneity of the equation \eqref{meqn}, without loss of generality, we assume that $\|u\|_{L^q(\Omega)}=1$. Let $k\geq 1$ and set $L(k):=\{x\in\Omega:u(x)>k\}$. Choosing $v=(u-k)^+$ as a test function in \eqref{mwksol}, we obtain
\begin{equation}\label{regtst1}
\begin{split}
&\int_{L(k)}|\nabla_H \,v|^p\,dx+\int_{\mathbb{H}^n}\int_{\mathbb{H}^n}{|v(x)-v(y)|^{p}}\,d\nu\,\\
&\leq\int_{L(k)}|\nabla_H\,u|^{p-2}\na_H\,u\,\na_H\,v\,dx+\int_{\mathbb{H}^n}\int_{\mathbb{H}^n}{J_p(u(x)-u(y))(v(x)-v(y))}\,d\nu\,\\
&=\lambda\int_{L(k)}u^{q-1}(u-k)\,dx.
\end{split}
\end{equation}
Therefore, we have
\begin{equation}\label{regtst2}
\int_{L(k)}|\nabla_H\,u|^p\,dx\leq\lambda\int_{L(k)}u^{q-1}(u-k)\,dx\leq\lambda\int_{L(k)}u^{q-1}(u-k)\,dx.
\end{equation}
We prove the result into two cases below.\\
\textbf{Case $I$.} Let $q\leq p$, then since $k\geq 1$, over the set $L(k)$, we have $u^{q-1}\leq u^{p-1}$. Therefore, from \eqref{regtst1} we have
\begin{equation}\label{regtst22}
\begin{split}
\int_{L(k)}|\nabla_H \,u|^p\,dx&\leq\lambda\int_{L(k)}u^{p-1}(u-k)\,dx\\
&\leq\lambda\int_{L(k)}(2^{p-1}(u-k)^{p}+2^{p-1}k^{p-1}(u-k))\,dx,
\end{split}
\end{equation}
where to obtain the last inequality above, we have used the inequality $(a+b)^{p-1}\leq 2^{p-1}(a^{p-1}+b^{p-1})$ for $a,b\geq 0$. Using Lemma \ref{embd} with $l=p$ in \eqref{regtst22} we obtain
\begin{equation}\label{regtst3}
\begin{split}
(1-S\lambda 2^{p-1}|L(k)|^\frac{p}{Q})\int_{L(k)}(u-k)^p\,dx&\leq\lambda S 2^{p-1}k^{p-1}|L(k)|^\frac{p}{Q}\int_{L(k)}(u-k)\,dx,
\end{split}
\end{equation}
where $S>0$ is the Sobolev constant. Note that $\|u\|_{L^1(\Omega)}\geq k|L(k)|$ and therefore for every $k\geq k_0=(2^p S\lambda)^\frac{Q}{p}\|u\|_{L^1(\Omega)}$, we have $S\lambda 2^{p-1}|L(k)|^\frac{p}{Q}\leq\frac{1}{2}$. Using this fact in \eqref{regtst3}, for every
$k\geq\max\{k_0,1\}$, we get
\begin{equation}\label{regtst4}
\begin{split}
\int_{L(k)}(u-k)^p\,dx&\leq \lambda S 2^{p}k^{p-1}|L(k)|^\frac{p}{\nu}\int_{L(k)}(u-k)\,dx.
\end{split}
\end{equation}
Using H\"older's inequality and the estimate \eqref{regtst4} we obtain
\begin{equation}\label{regtst5}
\int_{L(k)}(u-k)\,dx\leq (\lambda S2^p)^\frac{1}{p-1}k|L(k)|^{1+\frac{p}{Q(p-1)}}.
\end{equation}
Noting \eqref{regtst5}, by \cite[Lemma $5.1$]{LN}, we get $u\in L^\infty(\Omega)$. \\
\textbf{Case $II.$} Let $q>p$, then using the inequality $(a+b)^{q-1}\leq 2^{q-1}(a^{q-1}+b^{q-1})$ for $a,b\geq 0$ in \eqref{regtst1} we get
\begin{equation}\label{regtstc2}
\int_{L(k)}|\nabla_H\, u|^p\,dx\leq\lambda\int_{L(k)}(2^{q-1}(u-k)^{q}+2^{q-1}k^{q-1}(u-k))\,dx.
\end{equation}
Now, Lemma \ref{embd} with $l=q$ in the estimate \eqref{regtstc2} we obtain
\begin{equation}\label{regtstc21}
\begin{split}
\left(\int_{L(k)}(u-k)^q\,dx\right)^\frac{p}{q}&\leq S\lambda|L(k)|^{p(\frac{1}{q}-\frac{1}{p}+\frac{1}{Q})}\int_{L(k)}(2^{q-1}(u-k)^q+2^{q-1}k^{q-1}(u-k))\,dx,
\end{split}
\end{equation}
where $S>0$ is the Sobolev constant. Since $\int_{L(k)}(u-k)^q\,dx\leq\|u\|^{q}_{L^q(\Omega)}=1$ and $q>p$, the quantity in the left side of \eqref{regtstc21} can be estimated from below as

\begin{equation}\label{regtstc22}
\left(\int_{L(k)}(u-k)^q\,dx\right)^\frac{p}{q}=\left(\int_{L(k)}(u-k)^q\,dx\right)^{\frac{p-q}{q}+1}\geq\int_{L(k)}(u-k)^q\,dx.
\end{equation}

Using \eqref{regtstc22} in \eqref{regtstc21} we get
\begin{equation}\label{regtstc23}
\begin{split}
&\Big(1-S\lambda 2^{q-1}|L(k)|^{p(\frac{1}{q}-\frac{1}{p}+\frac{1}{Q})}\Big)\int_{L(k)}(u-k)^q\,dx\\
&\leq S\lambda 2^{q-1}k^{q-1}|L(k)|^{p(\frac{1}{q}-\frac{1}{p}+\frac{1}{Q})}\int_{L(k)}(u-k)\,dx.
\end{split}
\end{equation}
Let $\delta={p(\frac{1}{q}-\frac{1}{p}+\frac{1}{Q})}$, which is positive, since $1<q<p^{*}$. Choosing $k_1=(S\lambda 2^q)^\frac{1}{\delta}\|u\|_{L^1(\Omega)}$, due to the fact that $k|L(k)|\leq\|u\|_{L^1(\Omega)}$, we obtain for every $k\geq k_1$ that $S\lambda 2^{q-1}|L(k)|^\delta\leq\frac{1}{2}$. Using this property in \eqref{regtstc23}, we have

\begin{equation}\label{regtstc24}
\begin{split}
\int_{L(k)}(u-k)^q\,dx&\leq S\lambda 2^{q}k^{q-1}|L(k)|^\delta\int_{L(k)}(u-k)\,dx.
\end{split}
\end{equation}
By H\"older's inequality and the estimate \eqref{regtstc24} we arrive at
\begin{equation}\label{regtstc25}
\int_{L(k)}(u-k)\,dx\leq (\lambda S2^q)^\frac{1}{q-1}k|L(k)|^{1+\frac{\delta}{q-1}}.
\end{equation}
Noting \eqref{regtstc25}, by \cite[Lemma $5.1$]{LN}, we get $u\in L^\infty(\Omega)$.

\item[$(b)$] By \cite[Theorem 3.7]{GMBruno} (for $\alpha=0$), and \cite[Theorem $1.4$]{Mjde} (for $\alpha=1$), the result follows.
\end{enumerate}
\end{proof}

\section{Perturbed singular problem}
We introduce an approximation of the singular problem \eqref{meqn} and develop the corresponding variational framework. To this end, we follow the variational approach from \cite{Arcoya}. To this end, for $\e>0$, we consider the following approximated problem
\begin{equation*}
(P_{\mu,\alpha,\e})\qquad
\left\{
\begin{aligned}
 \mathcal{M}_\alpha\,u &= \frac{\mu}{(u^+ +\e)^{\gamma(x)}}+ (u^+)^\beta \quad \text{in } \Om,\\
 u&=0 \quad \text{in } \mathbb{H}^n \setminus \Om,
\end{aligned}
\right.
\end{equation*}
and establish existence of two distinct solutions by variational method. 
The energy functional $I_{\mu,\alpha,\e}:X_\alpha\to\mathbb{R}$ associated with $(P_{\la,\alpha,\e})$ is given by
\[
I_{\mu,\alpha,\e}(u)
= \frac{\|u\|^p}{p}
-\mu\int_\Om
\frac{(u^+ +\e)^{1-\gamma(x)}-\e^{1-\gamma(x)}}{1-\gamma(x)}\,dx
-\frac{1}{\beta+1}\int_\Om (u^+)^{\beta+1}\,dx.
\]
Finally to pass to the limit, we obtain a priori estimates on the approximate solutions, and obtain the multiplicity results.

\subsection{A priori estimates and convergence}
Let $I_{\mu,\alpha} : X_\alpha \to \mathbb{R}\cup\{\pm\infty\}$ denote the energy functional associated with problem \eqref{meqn}, defined by
\[
I_{\mu,\alpha}(u) = \frac{\|u\|^p}{p}
-\mu \int_\Om \frac{(u^+)^{1-\gamma(x)}}{1-\gamma(x)}\,dx
-\frac{1}{\beta+1}\int_\Om (u^+)^{\beta+1}\,dx.
\]

For $\e>0$, we consider the following approximated problem
\begin{equation*}
(P_{\mu,\alpha,\e})
\left\{
\begin{aligned}
 \mathcal{M}_\alpha\,u &= \frac{\mu}{(u^+ +\e)^{\gamma(x)}}+ (u^+)^\beta \quad \text{in } \Om,\\
 u&=0 \quad \text{in } \mathbb{H}^n \setminus \Om,
\end{aligned}
\right.
\end{equation*}

The energy functional $I_{\mu,\alpha,\e}:X_\alpha\to\mathbb{R}$ associated with $(P_{\la,\alpha,\e})$ is given by
\[
I_{\mu,\alpha,\e}(u)
= \frac{\|u\|^p}{p}
-\mu\int_\Om
\frac{(u^+ +\e)^{1-\gamma(x)}-\e^{1-\gamma(x)}}{1-\gamma(x)}\,dx
-\frac{1}{\beta+1}\int_\Om (u^+)^{\beta+1}\,dx.
\]
It is straightforward to verify that $I_{\mu,\alpha,\e}\in C^1(X_\alpha,\mathbb{R})$, 
$I_{\mu,\alpha,\e}(0)=0$, and $I_{\mu,\alpha,\e}(v)\leq I_{0,\alpha,\e}(v)$ for all 
$v\in X_\alpha$ with $v\geq 0$.

Without loss of generality, we assume that $\|e_{1,\alpha}\|_\infty=1$, where $e_{1,\alpha}$ is given by Lemma \ref{evpthm}. The following lemma shows that the functional $I_{\mu,\alpha,\e}$ satisfies the 
Mountain Pass geometry.

\begin{lemma}\label{MP-geo}
There exist constants $R>0$, $\rho>0$, and $\Lambda>0$ depending on $R$ such that
\[
\inf_{\|v\|\le R} I_{\mu,\alpha,\e}(v) < 0
\quad \text{and} \quad
\inf_{\|v\|=R} I_{\mu,\alpha,\e}(v) \ge \rho,
\quad \text{for all } \mu \in (0,\Lambda).
\]
Moreover, there exists $T>R$ such that
\[
I_{\mu,\alpha,\e}(T e_{1,\alpha}) < -1
\quad \text{for all } \mu \in (0,\Lambda).
\]
\end{lemma}
\begin{proof}
We prove the result by considering $\alpha=1$. The proof for $\alpha=0$ is exactly similar.
We fix 
\[
l := |\Om|^{\frac{1}{\left(\frac{p^*}{\beta+1}\right)'}}.
\]
Then, by H\"older's inequality together with Lemma \ref{embd}, we obtain
\begin{equation}\label{MP1}
\int_\Om (v^+)^{\beta+1}\,dx 
\le \left(\int_\Om |v|^{p^*}\,dx\right)^{\frac{\beta+1}{p^*}}
|\Om|^{\frac{1}{\left(\frac{p^*}{\beta+1}\right)'}}
\le Cl\|v\|^{\beta+1},
\end{equation}
for some positive constant $C$ independent of $v$.

Next, we observe that
\[
\lim_{t\to 0}\frac{I_{\mu,\alpha,\e}(t e_{1,\alpha})}{t}
= -\mu\int_{\Omega}\e^{-\gamma(x)}e_{1}\,dx < 0.
\]
Consequently, one can choose $k\in(0,1)$ sufficiently small and set
\[
R := k\left(\frac{\beta+1}{pCl}\right)^{\frac{1}{\beta+1-p}}
\]
so that
\[
\inf_{\|v\|\le R} I_{\mu,\alpha,\e}(v) < 0.
\]

Moreover, since 
\[
R < \left(\frac{\beta+1}{pCl}\right)^{\frac{1}{\beta+1-p}},
\]
we obtain
\[
I_{0,\alpha,\e}(v) \ge \frac{R^p}{p}-\frac{ClR^{\beta+1}}{\beta+1}
=: 2\rho > 0.
\]

We now define
\[
\Lambda := 
\frac{\rho}{\displaystyle 
\sup_{\|v\|=R}\left(
\frac{1}{1-\gamma(x)}
\int_\Om |v|^{1-\gamma(x)}\,dx
\right)},
\]
which is a positive constant. Since $\rho$ and $R$ depend on $k,\beta,p,|\Omega|,$ and $C$, the constant $\Lambda$ depends on these quantities as well.

Finally, we note the inequality
\begin{equation}\label{knownfact}
(v^{+}+\e)^{1-\gamma(x)}-\e^{1-\gamma(x)}
\le (v^+)^{1-\gamma(x)}.
\end{equation}
Using \eqref{knownfact}, we obtain
\[
I_{\mu,\alpha,\e}(v)
\geq \frac{\|v\|^p}{p}
-\frac{1}{\beta+1}\int_{\Om}(v^+)^{\beta+1}\,dx
-\frac{\mu}{1-\gamma(x)}\int_{\Om}(v^{+})^{1-\gamma(x)}\,dx
= I_{0,\alpha,\e}(v)
-\frac{\mu}{1-\gamma(x)}\int_{\Om}(v^{+})^{1-\gamma(x)}\,dx .
\]

Therefore,
\begin{align*}
\inf_{\|v\|=R} I_{\mu,\alpha,\e}(v)
&\geq \inf_{\|v\|=R} I_{0,\alpha,\e}(v)
-\mu \sup_{\|v\|=R}
\left(
\frac{1}{1-\gamma(x)}\int_\Om |v|^{1-\gamma(x)}\,dx
\right) \\
&\geq 2\rho
-\mu \sup_{\|v\|=R}
\left(
\frac{1}{1-\gamma(x)}\int_\Om |v|^{1-\gamma(x)}\,dx
\right)
\geq \rho,
\end{align*}
for all $\mu \in (0,\Lambda)$.

Finally, it is easy to see that $I_{0,\alpha,\e}(t e_{1,\alpha})\to -\infty$ as $t\to +\infty$. 
Hence we can choose $T>R$ such that $I_{0,\e}(T e_{1,\alpha})<-1$. Consequently,
\[
I_{\mu,\alpha,\e}(T e_{1,\alpha})\leq I_{0,\alpha,\e}(T e_{1,\alpha})<-1,
\]
which completes the proof.\qed
\end{proof}

Our next lemma shows that the functional $I_{\mu,\alpha,\e}$ satisfies the Palais--Smale $(PS)_c$ condition.

\begin{Proposition}\label{PS-cond}
The functional $I_{\la,\alpha,\e}$ satisfies the $(PS)_c$ condition for every $c\in \mathbb{R}$. 
That is, if $\{u_k\}_{k\in\mb{N}}\subset X_\alpha$ is a sequence such that
\begin{equation}\label{PS1}
I_{\la,\alpha,\e}(u_k)\to c 
\quad \text{and} \quad 
I_{\la,\alpha,\e}'(u_k)\to 0
\end{equation}
as $k\to\infty$, then $\{u_k\}_{k\in\mb{N}}$ admits a strongly convergent subsequence in $X_\alpha$.
\end{Proposition}

\begin{proof}
We prove the result for $\alpha=1$ only, since the proof for $\alpha=0$ follows by exactly similar arguments.

Let $\{u_k\}_{k\in\mathbb{N}}\subset X_\alpha$ be a sequence satisfying \eqref{PS1}. We claim that 
$\{u_k\}_{k\in\mathbb{N}}$ is bounded in $X_\alpha$. To see this, using \eqref{knownfact} we obtain
\begin{equation}\label{PS2}
\begin{split}
I_{\mu,\alpha,\e}(u_k)- \frac{1}{\beta+1}I_{\mu,\alpha,\e}'(u_k)u_k 
&= \left( \frac{1}{p}-\frac{1}{\beta+1}\right)\|u_k\|^p 
-{\mu}\int_\Om \frac{(u_k^+ +\e)^{1-\gamma(x)}-\e^{1-\gamma(x)}}{1-\gamma(x)}\,dx\\
&\quad +\frac{\mu}{\beta+1}\int_\Om (u_k^+ +\e)^{-\gamma(x)}u_k\,dx\\
&\geq \left( \frac{1}{p}-\frac{1}{\beta+1}\right)\|u_k\|^p 
-\mu\int_\Om \frac{(u_k^+)^{1-\gamma(x)}}{1-\gamma(x)}\,dx\\
&\quad + \frac{\mu}{\beta+1}\int_\Om (u_k^+ +\e)^{-\gamma(x)}u_k\,dx\\
&\geq \left( \frac{1}{p}-\frac{1}{\beta+1}\right)\|u_k\|^p 
-\mu\int_\Om \frac{(u_k^+)^{1-\gamma(x)}}{1-\gamma(x)}\,dx
-\frac{\mu C}{\e(\beta+1)}\|u_k\|,
\end{split}
\end{equation}
for some positive constant $C$ independent of $k$. Here we used the embedding 
result given in Lemma \ref{lemb} and the assumption $0<\gamma(x)<1$ in $\overline{\Om}$.

Using a similar argument, we also obtain
\begin{equation}\label{PS2-new}
\begin{split}
-\int_\Om \frac{(u_k^+)^{1-\gamma(x)}}{1-\gamma(x)}\,dx 
&\geq -\int_\Om \frac{|u_k|^{1-\gamma(x)}}{1-\gamma(x)}\,dx\\
&\geq \frac{-1}{1-\|\gamma\|_{\infty}}
\left(
\int_{\Om \cap \{|u_k|\ge 1\}} |u_k|^{1-\gamma(x)}\,dx
+\int_{\Om \cap \{|u_k|<1\}} |u_k|^{1-\gamma(x)}\,dx
\right)\\
&\geq \frac{-1}{1-\|\gamma\|_{\infty}}
\left(
\int_{\Om \cap \{|u_k|\ge 1\}} |u_k|\,dx
+\int_{\Om \cap \{|u_k|<1\}} 
|u_k|^{1-\|\gamma\|_{\infty}}\,dx
\right)\\
&\geq -C\left(\|u_k\|+\|u_k\|^{1-\|\gamma\|_{\infty}}\right),
\end{split}
\end{equation}
for some constant $C>0$ independent of $k$. Thus, inserting \eqref{PS2-new} into \eqref{PS2}, we obtain for some positive
constants $C_1$ and $C$ (both independent of $k$), that
\begin{equation}\label{PS2-new1}
I_{\mu,\alpha,\e}(u_k)- \frac{1}{\beta+1}I_{\mu,\alpha,\e}'(u_k)u_k
\ge C_1\|u_k\|^p
- C\left(\|u_k\|+\|u_k\|^{1-\|\gamma\|_{\infty}}\right).
\end{equation}

Moreover, from \eqref{PS1} it follows that for $k$ sufficiently large
\begin{equation}\label{PS3}
\left| I_{\mu,\alpha,\e}(u_k)- \frac{1}{\beta+1}I_{\mu,\alpha,\e}'(u_k)u_k \right|
\le c + o(\|u_k\|).
\end{equation}

Combining \eqref{PS2-new1} and \eqref{PS3}, we deduce that $\{u_k\}_{k\in\mathbb{N}}$ is bounded
in $X_\alpha$ since $p>1$. By the reflexivity of $X_\alpha$, there exists $u_0\in X_\alpha$ such that,
up to a subsequence,
\[
u_k \rightharpoonup u_0 \quad \text{weakly in } X_\alpha \quad \text{as } k\to\infty.
\]

\medskip
\noindent
\textbf{Claim.} $u_k \to u_0$ strongly in $X_\alpha$ as $k\to\infty$.

\medskip
For convenience, we define
\[
\mc A(v,\phi)
:=
\int_{\Om}
|\na_H\,v|^{p-2}\na_H\,v\,\na_H\,\phi\,dx,
\]
and
\[
\mc B(v,\phi)
:=
\int_{\mb H^{n}}\int_{\mb H^{n}}
J_p(v(x)-v(y))(\phi(x)-\phi(y))
\,d\nu.
\]

From \eqref{PS1} we obtain
\[
\lim_{k\to\infty}
\left(
\mc A(u_k,u_0)+\mc B(u_k,u_0)
-\mu\int_\Om (u_k^+ +\e)^{-\gamma(x)}u_0\,dx
-\int_\Om (u_k^+)^{\beta}u_0\,dx
\right)=0,
\]
and
\[
\lim_{k\to\infty}
\left(
\mc A(u_k,u_k)+\mc B(u_k,u_k)
-\mu\int_\Om (u_k^+ +\e)^{-\gamma(x)}u_k\,dx
-\int_\Om (u_k^+)^{\beta}u_k\,dx
\right)=0.
\]

Let $U_k(x,y)=u_k(x)-u_k(y)$ and $U_0(x,y)=u_0(x)-u_0(y)$. Then
\begin{equation}\label{PS4}
\begin{split}
&\lim_{k\to\infty}
\int_{\Om}(|\na_H\,u_k|^{p-2}\na_H\,u_k-|\na_H\,u_0|^{p-2}\na_H\,u_0)\na_H\,(u_k-u_0)\,dx\\
&\qquad+\lim_{k\to\infty}
\int_{\mb H^{n}}\int_{\mb H^{n}}
{
(J_p(U_k)-J_p(U_0))(U_k-U_0)
}\,d\nu\\
&=
\lim_{k\to\infty}
\Big(
\mu\int_\Om (u_k^+ +\e)^{-\gamma(x)}u_k\,dx
+\int_\Om (u_k^+)^\beta u_k\,dx
-\mu\int_\Om (u_k^+ +\e)^{-\gamma(x)}u_0\,dx
-\int_\Om (u_k^+)^\beta u_0\,dx
\Big) \\
&\quad
-\lim_{k\to\infty}\int_{\Om}(|\na_H\,u_0|^{p-2}\na_H\,u_0\,\na_H\,u_k-|\na_H\,u_0|^p)\,dx-\lim_{k\to\infty}
\big(\mc A(u_0,u_k)-\mc A(u_0,u_0)\big).
\end{split}
\end{equation}

From the weak convergence of $\{u_k\}_{k\in\mathbb{N}}$ in $X_\alpha$, we obtain
\begin{equation}\label{PS5}
\lim_{k\to\infty}
\big(\mc A(u_0,u_k)-\mc A(u_0,u_0)\big)=0.
\end{equation}

Indeed, $u_k\rightharpoonup u_0$ weakly in $X_\alpha$ implies
\[
\frac{u_k(x)-u_k(y)}{|y^{-1}\circ x|_{K}^{\frac{Q+sp}{p}}}
\rightharpoonup
\frac{u_0(x)-u_0(y)}{|y^{-1}\circ x|_{K}^{\frac{Q+sp}{p}}}
\quad \text{in } L^p(\mb H^{2n}),
\]
and
\[
\frac{J_p(u_0(x)-u_0(y))}
{|y^{-1}\circ x|_{K}^{\frac{Q+sp}{p'}}}
\in L^{p'}(\mb H^{2n}),
\]
which proves \eqref{PS5}.

Similarly, \(u_k \rightharpoonup u_0\) in \(X_\alpha\) implies
$
\nabla_H\,u_k \rightharpoonup \nabla_H\,u_0
\text{ weakly in } L^p(\Omega).
$
Since
\[
|\nabla_H\,u_0|^{p-2}\nabla_H\,u_0 \in L^{p'}(\Omega),
\]
one can obtain
\begin{equation}\label{llim}
\lim_{k\to\infty}
\int_{\Omega}
\left(
|\nabla_H\,u_0|^{p-2}\nabla_H\,u_0 \nabla_H\,u_k
-
|\nabla_H\,u_0|^p
\right)\,dx
=0.
\end{equation}

Next, since
\[
|(u_k^+ +\e)^{-\gamma(x)}u_0|
\le \e^{-\gamma(x)}|u_0|,
\]
and
\[
\int_\Om \e^{-\gamma(x)}|u_0|\,dx
\le (1+\e^{-\|\gamma\|_{\infty}})
\int_\Om |u_0|\,dx <\infty,
\]
Lebesgue's dominated convergence theorem yields
\begin{equation}\label{PS6}
\lim_{k\to\infty}
\int_\Om (u_k^+ +\e)^{-\gamma(x)}u_0\,dx
=
\int_\Om (u_0^+ +\e)^{-\gamma(x)}u_0\,dx .
\end{equation}

Similarly, by Vitali's convergence theorem we obtain
\begin{equation}\label{PS7}
\lim_{k\to\infty}
\mu\int_\Om (u_k^+ +\e)^{-\gamma(x)}u_k\,dx
=
\mu\int_\Om (u_0^+ +\e)^{-\gamma(x)}u_0\,dx .
\end{equation}

Using similar estimates, we also deduce
\begin{equation}\label{PS8}
\lim_{k\to\infty}
\int_\Om (u_k^+)^\beta u_0\,dx
=
\int_\Om (u_0^+)^\beta u_0\,dx,
\end{equation}
and
\begin{equation}\label{PS9}
\lim_{k\to\infty}
\int_\Om (u_k^+)^\beta u_k\,dx
=
\int_\Om (u_0^+)^\beta u_0\,dx.
\end{equation}

Substituting \eqref{PS5}--\eqref{PS9} into \eqref{PS4}, we obtain
\begin{equation}\label{klim}
\begin{split}
&\lim_{k\to\infty}
\int_{\Om}(|\na_H\,u_k|^{p-2}\na_H\,u_k-|\na_H\,u_0|^{p-2}\na_H\,u_0)\na_H\,(u_k-u_0)\,dx\\
&\qquad+\lim_{k\to\infty}
\int_{\mb H^{n}}\int_{\mb H^{n}}
{
(J_p(U_k)-J_p(U_0))(U_k-U_0)
}\,d\nu
=0.
\end{split}
\end{equation}

By Lemma \ref{alg}, it follows that both terms above are non-negative and, hence, each of the terms in the left hand side of \eqref{klim} above are zero. Moreover, by H\"older's inequality and Lemma \ref{alg1}, we have
\begin{align}
&\int_{\mathbb{H}^n}\int_{\mathbb{H}^n}
{\big(J_p(U_k)-J_p(U_0)\big)
      \big(U_k-U_0\big)}
     \,d\nu
\nonumber\\
&\quad\geq
\left(
[u_k]_{s,p}^{p-1}
-
[u_0]_{s,p}^{p-1}
\right)
\left(
[u_k]_{s,p}
-
[u_0]_{s,p}
\right)
\nonumber\\
&\quad\geq
C(p)\left(
[u_k]_{s,p}^{p/2}
-
[u_0]_{s,p}^{p/2}
\right)^2
\geq 0.
\end{align}
Similarly,
\begin{align}
&\int_{\Omega}
\big(
|\nabla_H\,u_k|^{p-2}\nabla_H\,u_k
-
|\nabla_H\,u_0|^{p-2}\nabla_H\,u_0
\big)
(\nabla_H\,u_k-\nabla_H\,u_0)\,dx
\nonumber\\
&\quad\geq
C(p)\left(
[u_k]_{1}^{p/2}
-
[u_0]_{1,p}^{p/2}
\right)^2
\geq 0.
\end{align}

Combining the above inequalities with \eqref{klim}, we obtain
$
\lim_{k\to\infty}\|u_k\|_{X_\alpha}
\longrightarrow
\|u_0\|_{X_\alpha}.
$
Since \(u_k\rightharpoonup u_0\) weakly in the uniformly convex Banach space
\(X_\alpha\), it follows that
$
u_k\longrightarrow u_0
\text{ strongly in }X_\alpha
\text{ as }k\to\infty.
$ \qed
\end{proof}

\begin{Remark}\label{mprmk1}
As a consequence of Lemma \ref{MP-geo}, we obtain
\[
\inf_{\|v\|=R} I_{\mu,\alpha,\e}(v) \ge \rho 
\max\{ I_{\mu,\alpha,\e}(T e_{1,\alpha}),\, I_{\mu,\alpha,\e}(0) \} = 0.
\]
\end{Remark}

\begin{Remark}\label{mprmk2}
By Lemma \ref{MP-geo} and Proposition \ref{PS-cond}, and the Mountain Pass Lemma,
it follows that for every $\mu\in(0,\Lambda)$ there exists $\zeta_\e \in X_\alpha$
such that $I_{\mu,\alpha,\e}'(\zeta_\e)=0$ and
\[
I_{\mu,\alpha,\e}(\zeta_{\e})
=\inf_{\theta\in\Gamma}\max_{t\in[0,1]} I_{\mu,\alpha,\e}(\theta(t))
\ge \rho>0,
\]
where
\[
\Gamma:=\{\theta\in C([0,1],X_\alpha):\theta(0)=0,\ \theta(1)=Te_{1,\alpha}\}.
\]

Using \eqref{knownfact} and \eqref{PS2-new} together with Vitali's convergence theorem, if
$u_k\rightharpoonup u_0$ weakly in $X_\alpha$, then
\[
\lim_{k\to\infty}
\int_{\Om}
\frac{(u_k+\e)^{1-\gamma(x)}-\e^{1-\gamma(x)}}{1-\gamma(x)}\,dx
=
\int_{\Om}
\frac{(u_0+\e)^{1-\gamma(x)}-\e^{1-\gamma(x)}}{1-\gamma(x)}\,dx .
\]
Consequently, the functional $I_{\mu,\alpha,\e}$ is weakly lower semicontinuous.

Moreover, as a consequence of Lemma \ref{MP-geo}, since for every
$\mu\in(0,\Lambda)$ we have
\[
\inf_{\|v\|\le R} I_{\mu,\alpha,\e}(v)<0,
\]
there exists a nonzero $\nu_\e\in X_\alpha$ with $\|\nu_\e\|\le R$ such that
\begin{equation}\label{limit-pass}
\inf_{\|v\|\le R} I_{\mu,\alpha,\e}(v)=I_{\mu,\alpha,\e}(\nu_\e)
<0<\rho\le I_{\mu,\alpha,\e}(\zeta_\e).
\end{equation}

Therefore, $\zeta_\e$ and $\nu_\e$ are two distinct nontrivial critical
points of $I_{\mu,\alpha,\e}$, whenever $\mu\in(0,\Lambda)$.
\end{Remark}

\begin{lemma}\label{non-negative}
The critical points $\zeta_\e$ and $\nu_\e$ of $I_{\mu,\alpha,\e}$ are nonnegative in $\Om$.
\end{lemma}

\begin{proof}
We test the equation $(P_{\mu,\alpha,\e})$ with $(\zeta_\e(x))_{-}=\min\{\zeta_\e(x),0\}$ and $(\nu_\e(x))_-=\min\{\nu_\e(x),0\}$, respectively. 
Since the right-hand side of $(P_{\mu,\alpha,\e})$ is nonnegative, it follows easily that 
$\zeta_\e \ge 0$ and $\nu_\e \ge 0$ in $\Om$.
\end{proof}

\begin{lemma}\label{apriori}
There exists a constant $\Theta>0$, independent of $\e$, such that
$\|v_\e\| \le \Theta$, where $v_\e=\zeta_\e$ or $\nu_\e$.
\end{lemma}

\begin{proof}
We prove the result only when $\alpha=1$, since the case $\alpha=0$ is similar. Indeed, the result is immediate when $v_\e=\nu_\e$. Hence we consider the case
$v_\e=\zeta_\e$. Recalling the notation from Lemma \ref{MP-geo}, define
\[
A:=\max_{t\in[0,1]} I_{0,\alpha,\e}(tTe_{1,\alpha}).
\]
Then
\[
A \ge \max_{t\in[0,1]} I_{\mu,\alpha,\e}(t\,Te_{1,\alpha})
\ge \inf_{\gamma\in\Gamma}\max_{t\in[0,1]} I_{\mu,\alpha,\e}(\gamma(t))
= I_{\mu,\alpha,\e}(\zeta_\e)
\ge \rho>0>I_{\mu,\alpha,\e}(\nu_\e).
\]
Consequently,
\begin{equation}\label{ap1}
\frac{1}{p}\|\zeta_\e\|^p
-\mu\int_\Om
\frac{(\zeta_\e+\e)^{1-\gamma(x)}-\e^{1-\gamma(x)}}{1-\gamma(x)}\,dx
-\frac{1}{\beta+1}\int_\Om \zeta_\e^{\beta+1}\,dx
\le A .
\end{equation}

Next, choosing $\phi=-\frac{\zeta_\e}{\beta+1}$ as a test function in $(P_{\mu,\alpha,\e})$,
we obtain
\begin{equation}\label{ap2}
-\frac{1}{\beta+1}\|\zeta_\e\|^p
+\frac{\mu}{\beta+1}\int_{\Om}\frac{\zeta_\e}{(\zeta_\e+\e)^{\gamma(x)}}\,dx
+\frac{1}{\beta+1}\int_{\Om}\zeta_\e^{\beta+1}\,dx
=0 .
\end{equation}

Adding \eqref{ap1} and \eqref{ap2}, we obtain
\begin{align*}
\left(\frac{1}{p}-\frac{1}{\beta+1}\right)\|\zeta_\e\|^p
&\le
\mu\int_\Om
\frac{(\zeta_\e+\e)^{1-\gamma(x)}-\e^{1-\gamma(x)}}{1-\gamma(x)}\,dx
-\frac{\mu}{\beta+1}\int_{\Om}\frac{\zeta_\e}{(\zeta_\e+\e)^{\gamma(x)}}\,dx
+A \\
&\le
\mu\int_\Om
\frac{(\zeta_\e+\e)^{1-\gamma(x)}-\e^{1-\gamma(x)}}{1-\gamma(x)}\,dx
+A \\
&\le
C\big(\|\zeta_\e\|+\|\zeta_\e\|^{1-\|\gamma\|_{\infty}}\big)+A,
\end{align*}
for some positive constant $C$ independent of $\e$, where the last
estimate follows from \eqref{PS2-new}, H\"older's inequality, and
Lemma \ref{lemb}. Since $\beta+1>p$, it follows that the sequence
$\{\zeta_\e\}_{\e>0}$ is uniformly bounded in $X_\alpha$ with respect to $\e$.
This completes the proof.
\end{proof}

As a consequence of Lemma \ref{non-negative} and Lemma \ref{apriori}, we obtain,
up to a subsequence, that $\zeta_\e \rightharpoonup \zeta_0$ and 
$\nu_\e \rightharpoonup \nu_0$ weakly in $X_\alpha$ as $\e \to 0^+$, for some
nonnegative functions $\zeta_0,\nu_0 \in X_\alpha$.

\begin{Remark}\label{2.24}
In the sequel, we show that $\zeta_0 \neq \nu_0$ and that they constitute
weak solutions of problem \eqref{meqn}. For convenience, we denote by
$v_0$ either $\zeta_0$ or $\nu_0$ and also recall that $v_\e=\zeta_\e$ or $\nu_\e$. To this end, to deal with the case $\alpha=1$, we require the following gradient convergence result (Lemma \ref{gradcgt}), which follows the lines of the proof of \cite[Theorem A.1]{GUnon} and \cite[Lemma 9]{Baldasfcaa}.
\end{Remark}

\begin{lemma}(Gradient convergence theorem)\label{gradcgt}
Let $\alpha=1$ and $v_0\in X_\alpha$ be as in Remark \ref{2.24}. Assume further that for every 
$\omega \Subset \Omega$, there exists a constant $c=c(\omega)>0$, such that
$
v_\e \ge c(\omega) \, \text{ in } \omega \, \text{ for all } \e .
$
Then, up to a subsequence,
$
\nabla_H\,v_\e \to \nabla_H\,v_0\text{ pointwise almost everywhere in } \Omega.
$
\end{lemma}

\begin{lemma}\label{Solution}
The function $v_0\in X_\alpha$ is a weak solution of problem \eqref{meqn}.
\end{lemma}

\begin{proof}
We prove the result for $\alpha=1$ only, since the proof is similar when $\alpha=0$. To this end, first we observe that for any $\e\in(0,1)$ and $t\ge 0$,
\[
\frac{\mu}{(t+\e)^{\gamma(x)}}+t^\beta
\ge \frac{\mu}{(t+1)^{\gamma(x)}}+t^\beta
\ge \min\left\{1,\frac{\mu}{2}\right\}.
\]
Consequently,
\[
\mathcal{M}_\alpha\,v_\e
=\frac{\mu}{(v_\e+\e)^{\gamma(x)}}+v_\e^\beta
\ge \min\left\{1,\frac{\mu}{2}\right\}
=:C .
\]

By Lemma \ref{auxresult}, there exists a solution $\xi\in X_\alpha\cap L^\infty(\Omega)$ of
\[
\mathcal{M}_\alpha\,\xi=C \quad \text{in }\Om, \qquad \xi>0 \quad \text{in }\Om
\]
such that for every $\omega\Subset\Omega$, there exists a constant $c(\omega)>0$ with $\xi\geq c(\omega)$ in $\omega$.
Then, 
\begin{equation}\label{strict positivity}
\begin{split}
&\int_{\Om}|\na_H\,v_\e|^{p-2}\na_H\,\xi\,\na_H\,\phi\,dx+
\int_{\mb H^{n}}\int_{\mb H^{n}}
J_p(v_{\e}(x)-v_{\e}(y))(\phi(x)-\phi(y))\,d\nu\\
&\ge
\int_{\Om}|\na_H\,\xi|^{p-2}\na_H\,\xi\,\na_H\,\phi\,dx+
\int_{\mb H^{n}}\int_{\mb H^{n}}
J_p(\xi(x)-\xi(y))(\phi(x)-\phi(y))\,d\nu
\end{split}
\end{equation}
for every nonnegative $\phi\in X_\alpha$. Choosing $\phi=(\xi-v_\e)^+\in X_\alpha$ as a test function in \eqref{strict positivity}
and using Lemma \ref{alg}, we deduce that
$
v_\e \ge \xi \text{ in }\Om .
$
Furthermore using the property of $\xi$ above,
there exists a constant $c=c(K)>0$ such that
\begin{equation}\label{uniform}
v_\e \ge c(K)>0 \qquad \text{for every } K\Subset\Om .
\end{equation}
Passing to the limit, we obtain $v_0\ge c(K)>0$ for every $K\Subset\Om$,
and hence $v_0>0$ in $\Om$.

Moreover, we observe that
\[
0\le \left|\frac{\mu\,\phi}{(v_\e+\e)^{\gamma(x)}}\right|
\le \mu\|\phi\,c(K)^{-\gamma(x)}\|_{\infty},
\qquad \forall\,\phi\in C_c^\infty(\Om).
\]
Therefore, by the Lebesgue's dominated convergence theorem,
\begin{equation}\label{limpass1}
\lim_{\e\to0^+}
\int_{\Om}\frac{\mu}{(v_\e+\e)^{\gamma(x)}}\phi\,dx
=
\int_{\Om}\frac{\mu}{v_0^{\gamma(x)}}\phi\,dx .
\end{equation}

Next, by the compact embedding given in Lemma \ref{lemb}, we have
\begin{equation}\label{limpass2}
\lim_{\e\to 0^+}\int_{\Om}v^{\beta}\phi\,dx
=
\int_{\Om}v_0^{\beta}\phi\,dx,
\qquad \forall\,\phi\in C_c^\infty(\Om).
\end{equation}

Moreover, by Lemma \ref{gradcgt}, we have 
\[
|\nabla_H\, v_\e|^{p-2}\nabla_H\, v_\e \rightharpoonup |\nabla_H\, v_0|^{p-2}\nabla_H\, v_0 
\text{ weakly in } L^{p'}(\Omega;\mathbb{R}^n).
\]
Since $\nabla \phi \in L^p(\Omega;\mathbb{R}^n)$, it follows that
\begin{equation}\label{llim1}
\lim_{\e\to 0^+} \int_{\Omega} 
|\nabla_H\, v_\e|^{p-2}\nabla_H\, v_\e \nabla_H\, \phi \, dx
=
\int_{\Omega} 
|\nabla_H\, v_0|^{p-2}\nabla_H\, v_0 \nabla_H\, \phi \, dx .
\end{equation}
Finally, using the weak convergence of $v_\e$ to $v_0$ in $X_\alpha$, we obtain for every $\phi\in C_c^\infty(\Om)$, that
\begin{equation}\label{limpass3}
\begin{gathered}
\lim_{\e\to 0^+}
\int_{\mb H^{n}}\int_{\mb H^{n}}
{J_p(v_\e(x)-v_\e(y))(\phi(x)-\phi(y))}
\,d\nu 
=
\int_{\mb H^{n}}\int_{\mb H^{n}}
{J_p(v_0(x)-v_0(y))(\phi(x)-\phi(y))}
\,d\nu.
\end{gathered}
\end{equation}

Combining \eqref{limpass1}--\eqref{limpass3}, we conclude that
\begin{align*}
&\int_{\Omega} 
|\nabla_H\, v_0|^{p-2}\nabla_H\, v_0 \nabla_H\, \phi \, dx+\int_{\mb H^{n}}\int_{\mb H^{n}}
{J_p(v_0(x)-v_0(y))(\phi(x)-\phi(y))}
\,d\nu\\
&=
\mu\int_{\Om}{v_0^{-\gamma(x)}}{\phi}\,dx
+
\int_{\Om}v_0^{\beta}\phi\,dx,
\end{align*}
for all $\phi\in C_c^\infty(\Om)$. Hence $v_0$ is a weak solution of
\eqref{meqn}.
\end{proof}

\subsection{Proof of the multiplicity result}
\begin{proof}(Proof of Theorem \ref{thm2})
Let $\alpha\in\{0,1\}$ and suppose that $u\in X_\alpha$ is a weak solution of \eqref{meqn}. Then, along the similar lines of the proof of \cite[Lemma 14]{GKRmathann}, it follows that $u\in L^\infty(\Om)$.

Let $\mu^*=\Lambda$, where $\Lambda$ is found in Lemma \ref{MP-geo} above. For the existence of multiple weak  solutions of \eqref{meqn}, we proceed as follows. First, using Lemma \ref{Solution}, we deduce that $\zeta_0$ and $\nu_0$ are two
positive weak solutions of \eqref{meqn} for $\mu\in(0,\Lambda)$.
We now show that $\zeta_0\neq \nu_0$.

Choosing $\phi=v_\e\in X_\alpha$ as a test function in $(P_{\mu,\alpha,\e})$, we obtain
\[
\|v_\e\|^p
=
\mu\int_{\Om}\frac{v_\e}{(v_\e+\e)^{\gamma(x)}}\,dx
+
\int_{\Om}v_\e^{\beta+1}\,dx .
\]

For $\alpha=1$, since $\beta+1<p^{*}$, it follows from Lemma \ref{lemb} that 
\[
\lim_{\e\to0^+}\int_{\Om}v_\e^{\beta+1}\,dx
=
\int_{\Om}v_0^{\beta+1}\,dx .
\]
Similarly, for $\alpha=0$, since $\beta+1<p_s^{*}$, by Lemma \ref{lemb}, the above estimate holds.

Moreover, observing that
\[
0\le \frac{v_\e}{(v_\e+\e)^{\gamma(x)}}\le v_\e^{1-\gamma(x)},
\]
and using \eqref{PS2-new} together with the Vitali's convergence theorem, we obtain
\[
\mu\lim_{\e\to0^+}\int_{\Om}\frac{v_\e}{(v_\e+\e)^{\gamma(x)}}\,dx
=
\mu\int_{\Om}v_0^{1-\gamma(x)}\,dx .
\]

Consequently,
\[
\lim_{\e\to0^+}
\|v_\e\|^p
=
\mu\int_{\Om}v_0^{1-\gamma(x)}\,dx
+
\int_{\Om}v_0^{\beta+1}\,dx.
\]

Using Lemma \ref{tlem}, choosing $\phi=v_0$ as a test function in
\eqref{meqn}, we obtain 
\[
\|v_0\|^p
=
\mu\int_{\Om}v_0^{1-\gamma(x)}\,dx
+
\int_{\Om}v_0^{\beta+1}\,dx.
\]

Therefore,
\[
\lim_{\e\to0^+}\|v_\e\|^p=\|v_0\|^p,
\]
which implies that $v_\e\to v_0$ strongly in $X_\alpha$.

Furthermore, by the Vitali's convergence theorem,
\[
\lim_{\e\to0^+}\int_{\Om}
\left[(v_\e+\e)^{1-\gamma(x)}-\e^{1-\gamma(x)}\right]\,dx
=
\int_{\Om}v_0^{1-\gamma(x)}\,dx .
\]

Combining this with the strong convergence of $v_\e$, we deduce that
\[
\lim_{\e\to0^+} I_{\mu,\alpha,\e}(v_\e)=I_{\mu,\alpha}(v_0).
\]
Hence, from \eqref{limit-pass}, it follows that $\zeta_0\neq \nu_0$. This completes the proof. \qed
\end{proof}

\section*{Perspectives and open problems}
 The following perspectives range from the technically concrete (uniqueness, critical exponents, variable order) to the conceptually expansive (control theory, numerical analysis, systems). Each represents a direction in which the present paper's framework could be extended, challenged, or enriched.

\subsection*{A new spectral landscape in sub-Riemannian geometry}
The paper opens a genuinely new chapter in the spectral theory of subelliptic operators. While the Euclidean theory of problems has begun to mature, the Heisenberg group setting introduces geometric obstructions, such as the absence of a global dilation structure compatible with the horizontal gradient, the anisotropy of the Kor\'anyi metric, and the noncommutativity of the group law, which render the Euclidean techniques insufficient. A natural perspective is to ask whether the inverse iteration scheme developed here can be adapted to more general stratified Lie groups (Carnot groups or M\'etivier groups), where the horizontal distribution is more intricate and the homogeneous dimension no longer determines the Sobolev exponents in the same way.

\subsection*{Beyond the sub-Laplacian: degenerate and anisotropic extensions}
The operators studied in  this paper are built from the \(p\)-sub-Laplacian and its fractional counterpart. A compelling open problem is to replace the horizontal gradient by a general H\"ormander-type system of vector fields satisfying the bracket-generating condition, or to consider weighted sub-Laplacians with Muckenhoupt weights. In such settings, the compact embedding \(X_\alpha \hookrightarrow L^q(\Omega)\) may fail or require a different proof, and the Sobolev critical exponents become weight-dependent.

\subsection*{Variable-order fractional diffusion}
The fractional exponent \(s \in (0,1)\) is fixed throughout the paper. A natural and increasingly important extension is to allow \(s\) to vary in space, that is, to consider the variable-order fractional sub-Laplacian \((-\Delta_{H,p})^{s(x)}\). This introduces a new layer of nonlocality whose interaction with the variable singular exponent \(\gamma(x)\) is entirely unexplored. The iteration scheme used in this paper would need to be redesigned, since the definition of the fractional operator changes pointwise.

\subsection*{Critical and supercritical singular nonlinearities}
The paper assumes \(\beta + 1 < p^*\) (or \(p_s^*\)), that is, the nonsingular term is subcritical. The critical case \(\beta + 1 = p^*\) and the supercritical case \(\beta + 1 > p^*\) remain open. In the critical case, the compact embedding is lost, and one must contend with concentration phenomena; in the supercritical case, variational methods break down entirely and one must resort to topological or bifurcation-theoretic arguments. The interaction of such critical growth with the singular term \(u^{-\gamma(x)}\) is likely to produce new multiplicity patterns and possibly nonexistence thresholds.

\subsection*{Uniqueness and nondegeneracy of the first eigenfunction}
The paper establishes the existence of a positive first eigenfunction  and proves that it is bounded and strictly positive inside the domain. However, uniqueness of the first eigenvalue and the simplicity of the first eigenfunction are not addressed. In the Euclidean \(p\)-Laplacian case, these properties are classical but delicate. In the mixed local-nonlocal Heisenberg setting, these problems are entirely open. A related question is whether the first eigenvalue is isolated and whether the eigenfunctions exhibit symmetry or monotonicity properties with respect to the Kor\'anyi metric.

\subsection*{Asymptotic behavior as \(s \to 1^-\) and \(s \to 0^+\)}
The parameter \(s \in (0,1)\) interpolates between two extreme regimes. On the one hand, as \(s \to 1^-\), the fractional sub-Laplacian converges (in a suitable sense) to the local sub-Laplacian, and the mixed operator becomes purely local. On the other hand, as \(s \to 0^+\), the nonlocal term degenerates and the problem becomes singularly perturbed. A systematic study of the limiting behavior of eigenvalues and eigenfunctions in these regimes would clarify the role of nonlocality and could reveal spectral bifurcation phenomena that are invisible at fixed \(s\).

\subsection*{Parabolic and evolution analogues}
This paper is entirely stationary. A natural perspective is to consider the associated parabolic mixed local-nonlocal problem on the Heisenberg group, where the spectral information obtained here would govern the asymptotic stability of equilibria and the large-time behavior of solutions. The fractional heat kernel on the Heisenberg group is notoriously difficult, and the mixed local-nonlocal case would require new heat kernel estimates compatible with the horizontal structure.

\subsection*{Numerical and computational perspectives}
The variational structure of the problem lends itself to numerical approximation, but the Heisenberg group geometry poses substantial challenges: the domain is typically unbounded in the vertical direction, the fractional operator is nonlocal, and the singularity at the origin requires adaptive mesh refinement. A computational study of the multiplicity result (visualizing the two distinct solutions and their dependence on \(\mu\)) would provide valuable intuition and could suggest new theoretical conjectures.

\subsection*{Connections to nonholonomic geometry and control theory}
A promising perspective is to interpret the eigenvalue problem as a spectral condition for controllability of a mixed local-nonlocal diffusion process on a sub-Riemannian manifold. The first eigenvalue, in particular, could serve as a threshold for stabilization or as a {measure of the cost of control in systems where diffusion is constrained to horizontal directions.

\subsection*{Generalization to systems and vector-valued problems}
The paper treats scalar equations. A natural extension is to systems of mixed local-nonlocal equations on the Heisenberg group, where the coupling between components introduces new spectral phenomena such as crossed eigenvalues and mode interactions. Such systems arise in models of multi-species diffusion, phase separation, and coupled oscillators on sub-Riemannian manifolds.

\section*{Acknowledgments}
Prashanta Garain acknowledges the financial support of the Anusandhan National Research
Foundation (ANRF), India under the ANRF Research Grant, File No. ANRF/ECRG/2024/000780/PMS. Vicen\c tiu D.~R\u adulescu was supported by the AGH University of
Krak\'{o}w under grant no. 16.16.420.054, funded by the Polish
Ministry of Science and Higher Education.

\end{document}